\documentclass[12pt]{amsart}
\usepackage{amsmath,amssymb,amsthm,fullpage,xcolor,ulem,tikz}
\usetikzlibrary{calc}
\usepackage[shortlabels]{enumitem}

\newcommand{\floor}[1]{\ensuremath{\left\lfloor #1\right\rfloor} }
\newcommand{\ceil}[1]{\ensuremath{\left\lceil #1\right\rceil} }
\newcommand{\size}[1]{\left \vert #1 \right \vert}

\newtheorem{lemma}{Lemma}[section]
\newtheorem{theorem}[lemma]{Theorem}
\newtheorem{proposition}[lemma]{Proposition}
\newtheorem{corollary}[lemma]{Corollary}

\newtheorem{conjecture}[lemma]{Conjecture}

\theoremstyle{definition}
\newtheorem{definition}[lemma]{Definition}

\DeclareMathOperator{\diam}{diam}
\DeclareMathOperator{\dist}{dist}
\DeclareMathOperator{\rad}{rad}

\def\cart{\, \Box \,}

\newcommand{\speeds}{c_{s,s}}
\newcommand{\speed}[1]{c_{#1,#1}}
\newcommand{\speedtwo}{c_{2,2}}
\newcommand{\capt}[1]{\mathrm{capt}_{#1,#1}}

\begin{document}
\title{Accelerated Cops and Robbers}

\author{William B. Kinnersley}
\address{Department of Mathematics and Applied Mathematical Sciences, University of Rhode Island, University of Rhode Island, Kingston, RI, USA, 02881}
\email{\tt billk@uri.edu}

\author{Nikolas Townsend}
%\address{Department of Mathematics and Applied Mathematical Sciences, University of Rhode Island, University of Rhode Island, Kingston, RI, USA, 02881}

%\email{\tt townsendn@uri.edu}

\address{Department of Mathematics, University of Massachusetts Boston, Wheatley Hall, 100 William T Morrissey Blvd, Boston, MA 02125}
\email{\tt nikolas.townsend@umb.edu}


\begin{abstract} We consider a variant of Cops and Robbers in which both the cops and the robber are allowed to traverse up to $s$ edges on each of their turns, where $s\ge 2$.  We give several general for this new model as well as establish bounds for the cop numbers for grids and hypercubes.  We also determine the capture time of cop-win graphs when $s=2$ up to a small additive constant. \end{abstract}

\maketitle

\begin{section}{Introduction}

%In this paper, we establish original results on the variant of the pursuit-evasion game Cops and Robbers where the speed of each player is increased from the original game.  Specifically, we consider the game where both the cops and the robber have the same speed.  This variant was originally introduced by Mehrabian \cite{MehThesis} in his master's thesis.

\textit{Cops and Robbers} is a pursuit-evasion game that has been studied extensively since its inception in the 1980s.  In this game, a team of cops aims to capture a robber on a given $G$.  To begin the game, the cops place themselves on vertices of the graph.  In response, the robber places himself down on a vertex.  From this point on, the game is played in \textit{rounds}, with each round consisting of a cops' turn and then a robber's turn.  During the cops' turn, each cop either moves to an adjacent vertex or to stays put; likewise, on the robber's turn, he either moves to an adjacent vertex or stays put.  The cops win if, at any point, some cop occupies the same vertex as the robber; when this happens, we say that the cops \textit{capture} the robber.  The robber wins if he can perpetually evade capture no matter how the cops choose to move.  For convenience, we will use she/her pronouns to refer an individual cop and he/him pronouns to refer to the robber.

Given a graph $G$, it is natural to ask for the minimum number of cops needed to capture a robber on $G$; we call this parameter the \textit{cop number} of $G$ and denote it by $c(G)$.  Cops and Robbers was first introduced independently by Quillot \cite{Qui78} and by Nowakowski and Winkler \cite{NW83}; both papers characterized the \textit{cop-win} graphs, i.e. graphs with cop number 1.  Later, Aigner and Fromme \cite{AF84} formally introduced the notion of cop number and published the first results on graphs with cop number greater than 1.  For more background on the history of Cops and Robbers, we refer the reader to \cite{BN11}.

Over the years, the game has spawned many variants.  Recently, variants of Cops and Robbers in which the robber moves at increased ``speed'' have garnered a substantial amount of attention.  We say that a robber has \textit{speed $s$} if he is allowed to traverse $s$ edges per turn, with the restriction that he cannot pass through any vertex occupied by a cop.  (Note that the case $s=1$ is equivalent to the original model of Cops and Robbers.)  This variant of the game has been studied by several authors.  Fomin, Golovach, Kratochv\'{i}l, Nisse, and Suchan \cite{FGKNS10} explored the computational complexity of this variant; they showed that if $s \ge 2$, then computing the cop number is NP-hard even on split graphs, and they showed that the parameter can be computed in polynomial time on graphs of bounded cliquewidth.  They also showed that $c_2(P_n \cart P_n) = \Omega(\log n)$, thereby showing that the cop number of a planar graph can be unbounded when $s \ge 2$.  More recently, Frieze, Krivilevich, and Loh \cite{FKL12} showed that when $s>1$, the number of cops needed to capture a robber on $G$ is at most $\displaystyle n/\left(1+1/s\right)^{1-o(1)\cdot \sqrt{\log_{1+1/s}(n)}}$.  The case where $s = \infty$ -- that is, where the robber can traverse as many edges as he likes on each turn -- has received considerable attention; see, for example, \cite{AM15}, \cite{KT22}, \cite{Meh12}, \cite{Meh15}, \cite{MehThesis}.

One can also increase the speed of the cops.  This idea was first suggested by Mehrabian \cite{MehThesis}, who proposed a model in which the cops and the robber both have the same speed $s$; he showed that $c_{s,s}(G)=c(G^s)$, where $G^s$ is the $s$th power of the graph $G$.  More generally, we can assign a speed $s$ to the cops and speed $t$ to the robber; we refer to this variant of Cops and Robbers as the \textit{speed-$(s,t)$} game, and we define the \textit{speed-$(s,t)$} cop number of a graph $G$, denoted $c_{s,t}(G)$, to be the minimum number of cops needed to guarantee capture of a robber in the speed-$(s,t)$ game.  In this paper, we will focus on the case where $s=t$, as first proposed by Mehrabian.% (see Theorem \ref{thm:graphpower} below).  We will make use of this result in Sections \ref{sec:general} and \ref{sec:capture_time}.

The speed-$(s,s)$ variant has connections to another version of Cops and Robbers, namely the \textit{distance-$k$} variant of the game.  In this variant, the rules are identical to the original game, except that the cops win if any cop occupies a vertex within distance $k$ of the robber.  Bonato, Chiniforooshan, and Pra\l{}at \cite{BCP10} introduced this variant and gave bounds on the maximum number of cops needed to capture a robber over all graphs of order $n$.  
%gave general bounds on the cop number of this game, denoted $c_{k}$.  In particular, they showed that $\displaystyle \left(\frac{n}{k}\right)^{1/2+o(1)}\le c_k(n)=O\left(\frac{n}{\log\left(\frac{2n}{k+1}\right)}\cdot\frac{\log(k+2)}{k+1}\right)$, where $c_k(n)$ is the maximum value of $c_k$ over all graphs of order $n$.  
Relating this variant to the speed-$(s,s)$ game, note that the cops can capture the robber in the speed-$(s,s)$ game only if, at the start of some cop turn, some cop is within distance $s$ of the robber; the same conditions permit the cops to win in the distance-$(s-1)$ game.  Hence, by mimicking a winning strategy for the speed-$(s,s)$ game, the cops can win in the distance-$(s-1)$ game.%  Thus, $c_{s-1}(G) \le \speeds(G)$ for any graph $G$.  A natural question (which is beyond the scope of this paper) would be ask for which graphs $G$ can we establish a lower bound on $c_{s-1}(G)$ in terms of $\speeds(G)$?

The overall structure of this paper is as follows.  In Section \ref{sec:general}, we provide several general results that relate the speed-$(s,s)$ game to other variants of Cops and Robbers.  We also investigate how $\speeds(G)$ can change as $s$ increases.
%specifically, we present significant progress to the claim that for any graph $G$, $c_{s,s}(G)$ is monotonically decreasing in $s$ (see Theorems \ref{thm:speeds_factors_monotone} and \ref{thm:nonincreasing_sequence}).  A consistent theme in the study of many of variants %in which some of the players are allowed to move faster than in the original game
% is the fascination in playing on Cartesian products of graphs; in particular, grid graphs (see \cite{NN98}, \cite{MehThesis}, \cite{FGKNS10}, \cite{KT22}).
In Section \ref{sec:Cartesian}, we investigate the speed-$(s,s)$ game on Cartesian products of graphs, with a focus on products of trees.  In particular, in Theorem \ref{thm:speed_s_large_grids} we show that $c_{s,s}(G)=\left\lceil d/2\right\rceil$ if $G$ is the $d$-fold Cartesian product of trees with large diameter, and in Theorem \ref{thm:hypercube_upper}, we give an upper bound on $\speeds$ on the $d$-dimensional hypercube.  
%We find in Corollary \ref{cor:speed_2_hypercube}, that this upper bound is tight when $s=2$, and we show a general lower bound on $c_{s,s}(Q_d)$ when $s>2$ in Theorem \ref{thm:speed_s_hypercube_lower_bound}. 
In Section \ref{sec:capture_time}, we will investigate the notion of capture time in the speed-$(s,s$) game.  The \textit{capture time} of a graph $G$ with cop number $k$ is the minimum number of rounds needed for $k$ cops to capture a robber on $G$, provided that the robber evades capture as long as possible.  This notion was introduced for the original model of Cops and Robbers by Bonato, Golovach, Hahn, and J. Kratochv\'{i}l \cite{BGHK09}, who showed that every cop-win graph of order $n$ has capture time at most $n-3$; Gaven\v{c}iak \cite{Gav10} later improved this bound to $n-4$, which is tight.  We show in Theorem \ref{thm:capture_time_speed2} that in the speed-$(2,2)$ game, the maximum capture time among all cop-win graphs of order $n$ is at least $n-7$. Finally, in Section \ref{sec:open_questions}, we present some future directions for further research.

%Aside from this result, no results on the speed-$(s,t)$ game have yet been published.  \comment{This sounds a bit negative.  Need to find a better way to frame this.  It reads as if people weren't able to get anyone to publish this stuff.}

%One might suppose that since the speed-$(s,s)$ game can, in some sense, be reduced to the original game by the previous result, that the study of this variant will not yield new or interesting results.  This connection to the original game is useful, however, we aim to show that the speed-$(s,s)$ game is interesting to study in its own right.

%In this paper, we will focus on the variant of the \textit{same-speed-$s$} variant of the game, wherein the cops and robber are each given a speed $s$, where $s\ge 2$.\comment{The preceding sentence is weird -- we already said all of this.}  
Before diving into Section \ref{sec:general}, we share some notation that we will use throughout the paper.  Unless otherwise specified, all graphs considered in this paper will be finite, undirected, and reflexive.  For a vertex $v$ in a graph $G$, we say that the set of all neighbors of $v$ is the \textit{neighborhood} of $v$, denoted $N[v]$; we say that the set of all vertices within distance $s$ of a vertex $v$ is the \textit{$s$th neighborhood of $v$}, denoted $N_s[v]$.  Note that $N_1[v]=N[v]$.  As a player traverses a path on their turn, we say that they take a \textit{step} each time that they move along an edge.  We use $\lg(x)$ to denote the base-2 logarithm of $x$.  
%Now we are ready to present general results about the speed-$(s,s)$ game.

\begin{section}{General Results}\label{sec:general}

In this section, we establish some elementary properties of the speed-$(s,s)$ game that will prove useful throughout the paper.  In addition, we will explore connections between the speed-$(s,s)$ game and the traditional model of Cops and Robbers (i.e. the speed-$(1,1)$ game).  We begin by recalling a result of Mehrabian \cite{MehThesis} that provides another perspective from which to view the speed-$(s,s)$ game. 

Let $\speeds(G)$ denote the speed-$(s,s)$ cop number of $G$.  Each player may, on their turn, move to any vertex within distance $s$ of their current location.  As Mehrabian shows, one may draw a parallel from the speed-$(s,s)$ game played on $G$ to the original version of Cops and Robbers played on $G^s$ (the graph with vertex set $V(G)$ and an edge joining vertices $u$ and $v$ if $\textrm{dist}_G(u,v) \le s$).

\begin{theorem}[\cite{MehThesis}, Theorem 10.1]\label{thm:graphpower}
For any graph $G$ and positive integer $s$, we have $\speeds(G) = c(G^s)$.
\end{theorem}

We will make extensive use of Theorem \ref{thm:graphpower} throughout the paper: in some circumstances it will be helpful to view the speed-$(s,s)$ game on $G$ from the perspective of having both players move at speed $s$, while in other circumstances it will be helpful to view it in terms of the ordinary version of Cops and Robbers played on $G^s$. 

Before stating the next result, we recall some relevant definitions.
\begin{definition}\label{def:retraction}
    A \textit{homomorphism} from a graph $G$ to a graph $H$ is a mapping $\phi \, : \, V(G) \rightarrow V(H)$ such that whenever $uv \in E(G)$, we have $\phi(u)\phi(v) \in E(H)$.  We further say that a homomorphism $\phi\, : \, V(G) \rightarrow V(H)$ is a \textit{retraction} from $G$ to $H$ if $H$ is a subgraph of $G$ and $\phi(v)=v$ for all vertices $v\in V(H)$.  When a retraction from $G$ to $H$ exists, we say that $H$ is a \textit{retract} of $G$.
\end{definition}

For the purposes of this paper, as is customary when studying Cops and Robbers, we assume that all graphs are reflexive.  Under this assumption, Berarducci and Intriglia \cite{BI93} showed that when $H$ is a retract of $G$, we have $c(G) \ge c(H)$.  Since then, the cop numbers of many variants of Cops and Robbers have likewise been shown to be monotone with respect to retraction; the speed-$(s,s)$ game is no exception.

\begin{theorem}\label{thm:retract}
For any graphs $G$ and $H$, if $H$ is a retract of $G$, then $\speeds(G) \ge \speeds(H)$.
\end{theorem}
\begin{proof}
We first claim that $H^s$ is a retract of $G^s$.  Let $\phi$ be a retraction from $G$ to $H$; we will show that $\phi$ is also a retraction from $G^s$ to $H^s$.  Since $\phi$ is a retraction from $G$ to $H$, and since $V(H) = V(H^s)$, we have $\phi(v)=v$ for all $v\in V(H^s)$.  It remains to show that $\phi$ is a homomorphism from $G^s$ to $H^s$.  Consider any edge $uv$ in $G^s$, and note that necessarily $\textrm{dist}_G(u,v)\le s$.  Let $\textrm{dist}_G(u,v)=i$ and let $u x_1x_2\dots x_{i-1}v$ be a path of length $i$ from $u$ to $v$ in $G$.  Since $\phi$ is a homomorphism from $G$ to $H$, $\phi(u)\phi(x_1)\phi(x_2)\dots\phi(x_{i-1})\phi(v)$ is a walk of length at most $i$ in $H$. Thus $\textrm{dist}_H(\phi(u),\phi(v))\le i\le s$, and $\phi(u)\phi(v)\in E(H^s)$, as desired.

Since $H^s$ is a retract of $G^s$, by the result of Berarducci and Intriglia we have $c(G^s) \ge c(H^s)$.  Additionally, by Theorem \ref{thm:graphpower}, we have $\speeds(G) = c(G^s)$ and $\speeds(H) = c(H^s)$.  So,
$$\speeds(G) = c(G^s) \ge c(H^s) = \speeds(H),$$
which completes the proof.
\end{proof}

In the speed-$(s,s)$ game where $s\ge 2$, all players are allowed to move faster than in the speed-$(1,1)$ game.  If each edge of $G$ were replaced with a path of length $s$, then it would take each player a full turn to move from vertex to vertex, as is the case in the original game.  Hence, one might suspect that the speed-$(1,1)$ game on this new graph behaves very similarly to the speed-$(s,s)$ game on $G$.  This is in fact the case, as we show next.

\begin{definition}
Given a graph $G$ and positive integer $s$, define $G^{(s)}$ to be the graph obtained from $G$ by subdividing each edge $s-1$ times.  The \textit{branch vertices} of $G^{(s)}$ are those vertices that are also present $G$; the \textit{subdivision vertices} are those vertices that were introduced in the course of subdividing edges.  Given adjacent vertices $u$ and $v$ in $G$, we say that a subdivision vertex is \textit{between} the branch vertices $u$ and $v$ in $G^{(s)}$ if it was introduced by subdividing edge $uv$.
\end{definition}

Intuitively, $G^{(s)}$ is obtained from $G$ by replacing each edge of $G$ with a path of length $s$.  For any $u,v \in V(G)$, we thus have $\textrm{dist}_{G^{(s)}} = s\cdot \textrm{dist}_G(u,v)$.  As previously mentioned, one might suspect that the speed-$(s,s)$ game played on $G$ is equivalent to the original game played on $G^{(s)}$.  This is nearly the case -- but not quite. 

\begin{theorem}\label{thm:subdivision}
    For every graph $G$ and positive integer $s$, we have $c(G)\le \speeds(G^{(s)})\le c(G)+1$.
\end{theorem}
\begin{proof}
    To show the lower bound, let $k$ speed-$s$ cops play on $G^{(s)}$, where $k<c(G)$.   Since $k < c(G)$, a speed-$1$ robber can evade $k$ speed-$1$ cops on $G$; we show how a speed-$s$ robber can mimic this strategy on $G^{(s)}$ to evade $k$ speed-$s$ cops.
    
    The robber will imagine that he is playing against $k$ cops on $G$, and he will use a winning strategy in the speed-$(1,1)$ game on $G$ to guide his play on $G^{(s)}$.  Throughout the game on $G^{(s)}$, the robber will remain solely on branch vertices.  He will begin the game on the branch vertex in $G^{(s)}$ dictated by his winning strategy on $G$ in the speed-$(1,1)$ game.  Henceforth, in each round he will move a full $s$ steps in order to end his turn on a branch vertex.  Specifically, if the robber's strategy on $G$ prescribes that he should move along edge $xy$ in $G$, then in $G^{(s)}$ the robber will move from $x$ along the subdivided edge (which is a path of length $s$) to $y$.  If a cop occupies a branch vertex $u$ in $G^{(s)}$, then the robber plays as if that cop occupies the vertex $u$ in $G$; if a cop occupies a subdivision vertex in $G^{(s)}$ between branch vertices $u$ and $v$, then the robber plays as if the cop occupies whichever vertex in $\{u,v\}$ the cop passed through most recently.  (If a cop begins the game on a subdivision vertex, then the robber arbitrarily chooses one of the corresponding branch vertices and pretends that the cop occupies that vertex.)
    
    The robber interprets the cop moves in $G^{(s)}$ and translates them to moves in the imagined game as follows.
    \begin{itemize}
        \item 
            If a cop moves from a branch vertex $u$ to a branch vertex $v$ along the subdivided edge between them in $G^{(s)}$, then the robber imagines that the cop moves from $u$ to $v$ in $G$.  %Clearly, this is a legal move in $G$, as the existence of a subdivided edge between $u$ and $v$ in  $G^{(s)}$ requires the existence of edge $uv$ in $G$.
        \item
            If a cop moves from a branch vertex $u$ to a subdivision vertex in $G^{(s)}$, then the robber imagines that the cop stays on the vertex $u$ in $G$.
        \item
            If the cop moves from a subdivision vertex through the branch vertex $u$ to another subdivision vertex, then the robber imagines that the cop moves to or stays on $u$ in $G$.
        \item
            Finally, if a cop moves from a subdivision vertex to a subdivision vertex but does not travel through a branch vertex, then the robber imagines that the cop stays put in $G$.
    \end{itemize}
%    Note that if the cop occupies a subdivision vertex between branch vertices $u$ and $v$ in $G^{(s)}$, then the only branch vertices the cop can reach on her turn are $u$ and $v$.  If the cop occupied $u$, then she would instead be able to reach all branch vertices corresponding to neighbors of $u$ in $G$.  Therefore, a cop occupying a subdivision vertex in $G^{(s)}$ can only prove to be less restrictive for the robber than a cop occupying the branch vertex that she passed through most recently.
    
    On each of the robber's turns, he moves to the branch vertex prescribed by his winning strategy against $k$ cops in $G$.  Since the robber can indefinitely evade $k$ speed-$1$ cops in the imagined game, he can indefinitely evade $k$ cops in the speed-$(s,s)$ game on $G^{(s)}$.
    %\textcolor{blue}{Do I need any more justification here?}

    For the upper bound, we show how $c(G)+1$ cops can capture a robber on $G^{(s)}$.  To do this, the cops will imagine a speed-$(1,1)$ game on $G$ and use a winning strategy for $c(G)$ cops in that game to capture the robber on $G^{(s)}$.

    Let $k=c(G)$, and label the cops $C_1, \dots C_k, C_{k+1}$.  Cops $C_1, \dots, C_k$ will mimic a winning cop strategy on $G$; cop $C_{k+1}$ will (for the moment) move arbitrarily.  For $1 \le i \le k$, if $C_i$ begins on vertex $v$ in the imagined game on $G$, then $C_i$ will begin on the branch vertex $v$ in $G^{(s)}$.  In response, the robber chooses a starting vertex in $G^{(s)}$.  Throughout the game, the cops will ensure that if the robber occupies a branch vertex $v$ in $G^{(s)}$, then he occupies $v$ in the imagined game on $G$; if instead the robber occupies a subdivision vertex between branch vertices $w$ and $x$ in $G^{(s)}$, then he must occupy either $w$ or $x$ in $G$.  To this end, if the robber begins on some branch vertex $w$, then the cops imagine that he occupies $w$ in $G$; if instead he begins on a subdivision vertex between branch vertices $w$ and $x$, then the cops arbitrarily choose either $w$ or $x$ and imagine that the robber occupies that vertex in $G$.

    Henceforth, each of the cops $C_1, \dots, C_k$ uses a winning strategy for $k$ cops in $G$ to guide their play in $G^{(s)}$.  In particular, if the cops' strategy on $G$ calls for $C_i$ to move from vertex $v$ to vertex $w$ in $G$, then $C_i$ moves from $v$ to $w$ in $G^{(s)}$; note that this move must be valid since $v$ and $w$ are distance $s$ apart in $G^{(s)}$.  In this way, cops $C_1, \dots, C_k$ always occupy the same vertices in the game on $G^{(s)}$ as they occupy in the imagined game on $G$.  When the robber moves in $G^{(s)}$, the cops imagine that he moves in the imagined game as follows:
    \begin{itemize}
    \item If the robber does not move to or through a branch vertex in $G^{(s)}$, then the cops imagine that he stays put in $G$.  This is clearly a legal robber move in $G$, and it clearly maintains the condition that if the robber occupies a subdivision vertex in $G^{(s)}$, then he occupies one of the corresponding branch vertices in $G$. 
    
    \item If the robber does move to or through a branch vertex $v$ in $G^{(s)}$, then the cops imagine that he moves to $v$ in $G$.  If the robber began on $v$ in $G^{(s)}$, then he began on $v$ in $G$, so his imagined move in $G$ is legal.  If the robber began on a subdivision vertex in $G^{(s)}$, then he moves to $v$ in $G$.  Since the branch vertices are distance $s$ apart in $G^{(s)}$, the robber can only pass through one branch vertex on a turn, so he must have begun his turn in $G$ on a vertex in $N[v]$; hence his imagined move to $v$ is legal.  Finally, if he began his turn on some other branch vertex $w$ in $G^{(s)}$, then he must have moved from $w$ to $v$ in $G^{(s)}$; consequently the distance in $G^{(s)}$ between $w$ and $v$ must be at most $s$, so we must have $vw \in E(G)$, hence the robber's move from $w$ to $v$ in $G$ is legal.  
    \end{itemize}
    
    Eventually, some cop $C$ captures the robber in the imagined game on $G$. If the robber occupies the corresponding branch vertex $v$ in $G^{(s)}$, then $C$ has captured the robber in $G^{(s)}$ and the cops have won.  Otherwise, the robber must occupy a subdivision vertex between $v$ and some other branch vertex $w$.  For the remainder of the game on $G^{(s)}$, cop $C$ takes $s$ steps toward the robber on each turn.  The remaining $k$ cops imagine that a new game has begun on $G$, in which the robber occupies vertex $w$.  As before, they imitate, on $G^{(s)}$, a winning strategy from $G$.  If the robber ever ends his turn on a branch vertex of $G^{(s)}$, then cop $C$ will capture him.  Moreover, the robber can never move to or through the branch vertex occupied by $C$.  Consequently, cop $C$ will always occupy a branch vertex $x$ in $G^{(s)}$, while in the imagined game on $G$, the robber will occupy some neighbor $y$ of $x$.  Since the cops are following a winning strategy for the imagined game on  $G$, some cop $C'$ eventually captures the robber in that game.  At this point, it must be that the robber occupies a subdivision vertex in $G^{(s)}$ between some branch vertices $x$ and $y$, while cops $C$ and $C'$ occupy both $x$ and $y$.  The robber cannot pass through $x$ or $y$ on his next turn, so on the ensuing cop turn, cop $C$ can capture him.
\end{proof}

We remark that both inequalities in Theorem \ref{thm:subdivision} are tight.  To see that the lower bound is tight, note that if $G$ is a tree, then so is $G^{(s)}$; hence $c(G) = c(G^{(s)}) = 1$.  Showing that the upper bound is tight requires a bit more work.

% \begin{proposition}
% \begin{enumerate}
% \item [\textbf{(a)}] For any positive integers $n$ and $s$ where $n\ge 3$ and $s\ge 2$, we have $c(K_n)=1$ and $\speeds(K_n^{(s)}) = 2$.
% \item [\textbf{(b)}] *** ADD A FAMILY OF EXAMPLES WHERE WE HAVE EQUALITY, PREFERABLY FOR ALL POSSIBLE COP NUMBERS ***
% \end{enumerate}
% \end{proposition}
\begin{proposition}\label{prop:complete_subdivision}
For any positive integers $n$ and $s$ where $n\ge 3$ and $s\ge 2$, we have $c(K_n)=1$ and $\speeds(K_n^{(s)}) = 2$.
\end{proposition}
\begin{proof}
It is clear that $c(K_n) = 1$, so by Theorem \ref{thm:subdivision}, we need only show that $\speeds(K_n^{(s)}) > 1$.  Toward this end, we begin by showing that for any $n\ge 4$, the graph $K_{n-1}^{(s)}$ is a retract of $K_n^{(s)}$; this will allow us to reduce to the case where $n=3$.  Let $G = K_n^{(s)}$, fix any two branch vertices $u$ and $v$ in $G$, and let $H = G-u$; note that this implies $H = K_{n-1}^{(s)}$.  For any other branch vertex $w$ in $G$, let the subdivision vertices corresponding to edge $uw$ be $(uw)_1, (uw)_2, \dots, (uw)_{s-1}$, in order, with $(uw)_1$ adjacent to $u$ and $(uw)_{s-1}$ adjacent to $w$; likewise, let the subdivision vertices corresponding to edge $vw$ be $(vw)_1, (vw)_2, \dots, (vw)_{s-1}$, in order, with $(vw)_1$ adjacent to $v$ and $(vw)_{s-1}$ adjacent to $w$.  Define the map $\phi: V(G)\to V(H)$ as follows:
\begin{enumerate}[label=(\alph*)]
    \item
        $\phi(u)=v$;
    \item
        $\phi(x) = v$ for every subdivision vertex $x$ between $u$ and $v$;
    \item
        $\phi((uw)_i) = (vw)_i$ for every branch vertex $w$ with $w \not \in \{u,v\}$ and for all $i \in \{1, \dots, s-1\}$;
    \item
        $\phi(x)=x$ for all other vertices $x$.
\end{enumerate}

We claim that $\phi$ is a retraction.  Property (d) guarantees that $\phi$ is the identity map on $H$, so we need only show that $\phi$ is a homomorphism.  Let $xy \in E(G)$.  If $x,y \in V(H)$, then $\phi(x)\phi(y) = xy \in E(H)$, so we need only consider edges for which at least one endpoint does not belong to $G$.  Thus, we may assume without loss of generality that $x$ is a subdivision vertex between $u$ and some other branch vertex of $G$.  If $x$ lies between $u$ and $v$, then both $x$ and $y$ map to $v$, so the edge $xy$ is preserved under $\phi$.  (Recall that we are working only with reflexive graphs, so $vv \in E(H)$.)  Suppose instead that $x$ lies between $u$ and $w$ for some branch vertex $w$ with $w \not\in \{u,v\}$.  The edge $xy$ must belong to the path $u(uw)_1(uw)_2\dots (uw)_{s-1}w$ in $G$; this path maps, under $\phi$, to $v(vw)_1(vw)_2\dots (vw)_{s-1}w$, so once again edge $xy$ is preserved.  Since $\phi$ preserves all adjacencies it is a homomorphism and hence a retraction, as claimed.

Repeated application of the argument above shows that $K_3^{(s)}$ is a retract of $K_n^{(s)}$, so Theorem \ref{thm:retract} implies that $\speeds(K_n^{(s)}) \ge \speeds(K_{3}^{(s)})$ for all $n\ge 3$.  To complete the proof, it now suffices to show that $\speeds(K_n^{(s)}) > 1$.  

Note that $K_3^{(s)}\cong C_{3s}$.  Once the cop places herself on a vertex of the graph, the robber chooses to occupy a vertex at distance $s+1$ from the cop.  This ensures that the robber cannot be caught on the first cop turn.  If the cop's move puts her on a vertex at distance $s+1$ or more from the robber, then the robber remains at his current vertex and thus avoids capture on the ensuing cop turn.  Otherwise, the cop ends her turn at a vertex that is distance $x$ from the cop, where $1\le x\le s$.  The robber can then move $s+1-x$ steps around the cycle away from the cop to land on a vertex that is distance $s+1$ from the cop and, once again, avoid capture on the ensuing cop turn.  In this way, the robber can evade the cop indefinitely.
\end{proof}

Although Proposition \ref{prop:complete_subdivision} shows that we can have $\speeds(G^{(s)}) = c(G)+1$ when $c(G) = 1$, we are unaware of any graphs $G$ with  $\speeds(G^{(s)}) = c(G)+1$ and $c(G) \ge 2$. 
\\
%This observation fuels our first conjecture, which we will discuss further in Section \ref{sec:open_questions}.\comment{We don't want to say ``this will be discussed in Section whatever''; just state the conjecture and move on.  More importantly, I don't know that we want to conjecture this.  We can present it as an open problem, but I don't know about making it a conjectuve.  We actually might not even want to mention it here; we can save it for the end of the paper.  I worry about drawing too much attention to things we don't know.}
%\begin{conjecture} \label{conj:subdivision_copnumber}
%    If $c(G)>1$, then $c_{s,s}(G^{(s)})=c(G)$.
%\end{conjecture}

For our next result, we establish a relationship between the speed-$(s,s)$ game and a well-studied variant of Cops and Robbers originally introduced by Aigner and Fromme \cite{AF84}.  The \textit{active} variant is played identically to the original game except that on every cop turn, at least one cop must move to a neighboring vertex, and on every robber turn, the robber must move to a neighboring vertex.  (That is, the cops cannot all remain in place on a cop turn, and the robber cannot remain in place on a robber turn.)  The minimum number of cops needed to capture a robber on a graph $G$ in this variant of the game is denoted by $c'(G)$.  To facilitate the proof, we actually consider yet \textit{another} variant of Cops and Robbers: the \textit{semi-active} variant is the same as the active variant, except that we no longer require any cops to move on a cop turn (though the robber is still required to move on every robber turn).  We denote the minimum number of cops needed to capture a robber on $G$ in the semi-active variant by $c^-(G)$.  

\begin{theorem}\label{thm:speed_s_active}
For every graph $G$ and for every $s \ge 2$, we have $\speed{s}(G) \le c^-(G) \le c'(G)$.
\end{theorem}
\begin{proof}
It is clear that $c^-(G) \le c'(G)$, since the cops have more options in the semi-active game than in the active game.  For the other inequality, let $k = c^-(G)$; we give a strategy for $k$ cops to capture a robber in the speed-$(s,s)$ game on $G$.  

Essentially, the cops view each robber move in the speed-$(s,s)$ game as a sequence of moves in the semi-active game and respond to each one in turn.  More formally, the cops imagine that they are playing the semi-active game on $G$ and use a winning strategy in that game to capture the robber in the speed-$(s,s)$ game.  The cops choose their initial positions in the semi-active game according to a winning strategy for that game, and they begin in the same positions in the speed-$(s,s)$ game.  Similarly, once the robber chooses an initial position in the speed-$(s,s)$ game, the cops imagine that he has begun on the same vertex in the semi-active game. In the speed-$(s,s)$ game, the cops skip their first turn, i.e. all cops remain in place; in the imagined semi-active game, they skip their first turn.  (Note that this does not adversely affect their strategy for the semi-active game, since it effectively just forces the robber the opportunity to choose a new starting position.)  

After the robber makes a move in the speed-$(s,s)$ game, the cops respond as follows.  Suppose first that the robber does not remain in place, i.e. that he moves from his current vertex $v$ to a new vertex $w$, and let $v=v_0, v_1, \dots, v_d = w$ be a cop-free path from $v$ to $w$ of length at most $s$.  In the semi-active game, the cops imagine that the robber moves from $v_0$ to $v_1$, after which the cops respond using a winning strategy in the semi-active game.  They then imagine that the robber moves from $v_1$ to $v_2$, and they respond accordingly.  This continues for $d$ rounds, after which point the robber has arrived at $w$ in the semi-active game and the cops have responded to each of his moves in turn.  Note that each cop has moved at most $s$ steps in the semi-active game.  Therefore, in the speed-$(s,s)$ game, each cop can (and does) move to her position in the imagined semi-active game.  If instead the robber chooses to remain in place in the speed-$(s,s)$ game, then the cops play similarly, except that they imagine that, in the semi-active game, the robber moves from $v$ to some neighbor $v'$ and then moves from $v'$ back to $v$ over the course of two turns.  Since $s\ge 2$, each cop can respond to each of these steps in the speed-$s$ game according to their winning strategy in the semi-active game.

Thus, in each round of the speed-$(s,s)$ game, the cops simulate one or more rounds of the semi-active game.  Since the cops follow a winning strategy in the semi-active game, some cop eventually captures the robber.  Henceforth, that cop can mirror the robber's movements, thereby ensuring that after every cop move in the semi-active game, some cop occupies the same vertex as the robber.  Consequently, after the corresponding cop move in the speed-$(s,s)$ game, that cop will likewise occupy the same vertex as the robber; hence, the cops win the speed-$(s,s)$ game. 
\end{proof}

Next, we consider the question of how $\speeds(G)$ can change as $s$ increases.  Trivially, $\speed{1}(G) = c(G)$ for all graphs $G$.  It is also easy to see that when $s \ge \rad(G)$ we have $\speeds(G) = 1$, since a single cop beginning on a central vertex can capture the robber on her first turn.  Thus as $s$ increases, $\speeds(G)$ starts off equal to $c(G)$ -- which may be quite large -- and eventually decreases to 1.  What happens in between?

In the ``long run'' for any graph $G$ -- as $s$ goes from 1 up to $\textrm{rad}(G)$ -- we have that $\speeds(G)$ is nonincreasing in $s$.  While it may be natural to assume that $\speeds$ is monotonically nonincreasing in $s$, proving this appears to be difficult (see Conjecture \ref{conj:monotone}).  However, we next provide a result showing that $\speeds$ is ``mostly'' nonincreasing.

\begin{theorem}\label{thm:speeds_factors_monotone}
For any positive integers $k$ and $s$, we have $\speeds(G) \ge \speed{ks}(G)$.
\end{theorem}
\begin{proof}    
    We will show that $\speeds(G)$ cops can capture the robber in the speed-$(ks,ks)$ game on $G$.  Similarly to the strategy used in Theorem \ref{thm:speed_s_active}, the cops will view the robber's $ks$ steps as $k$ consecutive moves made by a speed-$s$ robber and respond to each according to their strategy in the speed-$(s,s)$ game.  
    
    The cops start on vertices dictated by their winning strategy in the speed-$(s,s)$ game.  After the robber chooses his starting vertex, the cops skip their first turn.  As in the proof of Theorem \ref{thm:speed_s_active}, this does not affect the cop strategy for the speed-$(s,s)$ game, since it essentially gives the robber the opportunity to re-select his starting position, if he wishes.  Consider a robber turn, and suppose that the robber moves to a vertex distance $t$ away (where necessarily $0\le t\le ks$).  If $t<ks$, the cops may interpret this move by assuming that the robber stays put for $ks-t$ turns after taking the initial $t$ steps.  Thus, on every turn, the cops may assume that the robber always takes exactly $ks$ steps on each turn, with some steps consisting of the robber standing still. After the robber moves in the speed-$(ks,ks)$ game, the cops respond as follows.  The cops imagine that the robber's $ks$ steps consist of $k$ consecutive moves in the speed-$(s,s)$ game and respond to each of these moves in turn according to their winning strategy for the speed-$(s,s)$ game.  Since the cops follow a winning strategy in the speed-$(s,s)$ game, some cop eventually captures the robber.  Thereafter, that cop can mirror the robber's movements, ensuring that at the end of every cop move in the speed-$(s,s)$ game, some cop occupies the same vertex as the robber.  Thus, after the corresponding move in the speed-$(ks,ks)$ game, that cop will occupy the same vertex as the robber and will capture him.
\end{proof}

The following corollary follows immediately from Theorem \ref{thm:speeds_factors_monotone} and the fact that $c(G) = \speed{1}(G)$.

\begin{corollary}
For any positive integer $s$, we have $\speeds(G) \le c(G)$.
\end{corollary}

Theorem \ref{thm:speeds_factors_monotone} reinforces our earlier observation that as $s$ increases, $\speeds(G)$ decreases, at least gradually.  As mentioned above, we suspect that $\speeds(G)$ is in fact monotone in $s$, but a proof has so far eluded us.

\begin{conjecture}\label{conj:monotone}
For every graph $G$ and integers $s$ and $s'$ with $s < s'$, we have $\speed{s}(G) \ge \speed{s'}(G)$.
\end{conjecture}

Another interpretation of Conjecture \ref{conj:monotone} is that for any graph $G$, the sequence $\speed{1}(G), \speed{2}(G), \dots$ is nonincreasing and eventually reaches 1 (since $\speeds(G) = 1$ whenever $s \ge \textrm{rad}(G)$).  We next prove a sort of converse of this statement, namely that any nonincreasing positive integer sequence that eventually reaches 1 can be realized as the sequence $\speed{1}(G), \speed{2}(G), \speed{3}(G), \dots$ for some graph $G$.  We begin with a lemma that provides a useful building block for constructing such a graph $G$.

\begin{lemma} \label{lem:speed_s_graph_building}
Fix positive integers $k$ and $s$, and let $G$ be the $k$-fold strong product of $C_{2s+2}$.  For any positive integer $s'$, we have
$$\speed{s'}(G) = \left\{\begin{array}{ll}k+1, \quad &\text{if } s' \le s\\1, &\text{otherwise}\end{array}\right .$$
\end{lemma}
\begin{proof}
We view the vertex set of $G$ as the set of ordered $k$-tuples $(x_1, x_2, \dots, x_k)$ where $0 \le x_i \le 2s+1$ for each $i$.  Note that two vertices are adjacent in $G$ if and only if they differ by at most 1 (modulo $2s+2$) in each coordinate.

$G$ has diameter $s+1$, since any two vertices can differ by at most $s+1$ (modulo $2s+2$) in each coordinate.  It follows that when $s' > s$ we have $\speed{s'}(G) = 1$, since a single cop can reach any vertex of $G$ in a single turn. 

Suppose instead that $s' \le s$.  For the upper bound, a result of Neufeld and Nowakowski (\cite{NN98}, Theorem 4.1) implies that for all graphs $H_1, H_2, \dots, H_k$ with $c(H_1) \ge 2$, we have $c(\boxtimes_{i=1}^k H_i) \le \sum_{i=1}^k c(H_i) -k+1$.  It is straightforward to see that $c(C_{2s+2}^s) \ge 2$, since the robber can avoid a single cop on $C_{2s+2}^s$ by always remaining on the unique vertex distance $s+1$ from the cop.  Noting that $G^s = \boxtimes_{i=1}^k C_{2s+2}^s$, we have
$$\speed{s}(G) = c(G^s) = c\left (\boxtimes_{i=1}^k C_{2s+2}^s\right) \le \sum_{i=1}^k c(C_{2s+2}^s) - k + 1 \le 2k-k+1 = k+1.$$

% The second equality above can be justified by showing that $G^s=\boxtimes_{i=1}^k C^{s}_{2s+2}$.  To see this, let $H=\boxtimes_{i=1}^k C^{s}_{2s+2}$ and note that $V(G^s)=V(H)$.  It remains to show that $E(G^s)=E(H)$.  \textcolor{blue}{FINISH THIS}.

For the lower bound, we give a strategy for the robber to evade $k$ cops on $G$.  Label the cops $C_1, C_2, \dots, C_k$, and for all $i,j \in \{1, \dots, k\}$, denote the $j$th coordinate of $C_i$'s initial position by $p_{i,j}$.  The robber begins the game on the vertex $(p_{1,1}+s+1, p_{2,2}+s+1, \dots, p_{k,k}+s+1)$, where coordinates are taken modulo $2s+2$.  This ensures that after the cops move, $C_i$'s new position has a different $i$th coordinate from the robber's position for all $i \in \{1, \dots, k\}$; consequently, the cops cannot capture the robber on their first turn, as each $C_i$ will have a different $i$th coordinate from the robber.  On the robber's turn, he then moves to the vertex whose $i$th coordinate differs by $s+1$ from the $i$th coordinate of $C_i$'s position (with coordinates taken modulo $2s+2$) for all $i \in \{1, \dots, k\}$.  Once again, this leaves the cops unable to capture the robber on their next turn.  The robber can repeat this process indefinitely, thereby perpetually evading the cops.
\end{proof}

Armed with Lemma~\ref{lem:speed_s_graph_building}, we next show that any nonincreasing positive integer sequence that eventually reaches 1 can be realized as the sequence $\speed{1}(G), \speed{2}(G), \speed{3}(G), \dots$ for some graph $G$. 

\begin{theorem} \label{thm:nonincreasing_sequence}
    For every nonincreasing sequence of positive integers $t_1,t_2,\dots$ whose terms eventually reach 1, there exists a graph $G$ such that $\speed{s}(G)=t_s$ for all positive integers $s$.
\end{theorem}
\begin{proof}
    Note that since $t_1$ is finite and all terms of the sequence are positive integers, there can be at most $t_1-1$ indices $i$ for which $t_i > t_{i+1}$.  Let $i_0 = 0$, let $i_1,i_2,\dots,i_\ell$ denote the indices for which $t_{i_j}>t_{i_{j}+1}$, and note that by assumption $t_m = 1$ whenever $m > i_{\ell}$.  For each $k\in \{1,\dots,\ell\}$, let $H_k$ denote the $(t_{i_k}-1)$-fold strong product of $C_{2i_k+2}$.  Note that by Lemma \ref{lem:speed_s_graph_building}, for all $k \in \{1, \dots, \ell\}$ we have $\speeds(H_k) = t_{i_k}$ for $s \le i_{k}$ and $\speeds(H_k) = 1$ for $s > i_{k}$.  Construct a graph $G$ by taking $H_1, H_2, \dots, H_{\ell}$ and adding a new vertex $v$ that is adjacent to vertex $(0, \dots, 0)$ in each of the $H_i$.
    %; call $v$ the \textit{central vertex} of $G$.  
    
    We claim that $\speeds(G)=t_s$ for all positive integers $s$. Fix $s$.  To see that $\speeds(G) \le t_s$, suppose $t_s$ speed-$s$ cops play against a speed-$s$ robber on $G$.  Suppose first that $s \le i_{\ell}$, and fix $k$ such that $i_{k-1} < s \le i_{k}$.  Since $s \le i_{k}$, as noted above we have $\speeds(H_k) = t_{i_k}$.  Moreover, for any $k'< k$ we have $s > i_{k'}$, hence $\speeds(H_{k'}) = 1$; for any $k' > k$ we have $s \le i_{k'}$, so $\speeds(H_{k'}) = t_{i_{k'}} < t_{i_{k}}$. In other words, we have $\speeds(H_k) = t_{i_k} = t_s$ and, for all $k' \not = k$, we have $\speeds(H_{k'}) < t_{i_{k}}$.    
    
    The cops begin the game on any vertices in $H_k$.  They will play according to a winning strategy in the speed-$s$ game on $H_k$; to facilitate this, if the robber ever occupies a vertex outside $H_k$, the cops play as if the robber occupies the vertex $(0,0,0,\dots)$ in $H_k$.  Since the cops use a winning strategy in $H_k$, they either capture the robber or reach a configuration in which some cop $C$ occupies $(0,0,0,\dots,0)$ in $H_k$ and the robber occupies a vertex outside of $H_k$.  On the cops' next turn, $C$ will move to $v$ and remain there for the remainder of the game.  The robber must now occupy a vertex in $H_{k'}$ for some $k'\in\{1,\dots,\ell\}-\{k\}$.  Since $C$ guards the central vertex $v$, the robber will be confined to $H_{k'}$ for the rest of the game.  The remaining $t_{i_k}-1$ cops other than $C$ gradually move to $H_{k'}$.  As noted above, we have $\speeds(H_{k'}) < t_{i_k} = t_s$, hence $\speeds(H_{k'}) \le t_{i_k}-1$, so the cops in $H_{k'}$ can follow a winning strategy for the speed-$s$ game on $H_{k'}$ and thereby capture the robber.

    Now consider the case where $s > i_{\ell}$; we must show that $\speeds(G) = t_s = 1$.  The cop will begin on the central vertex $v$; suppose the robber begins on a vertex in some $H_k$.  The diameter of $H_k$ is $i_{k}+1$, which is at most $i_{\ell}+1$ and hence at most $s$.  On the cop's first turn, she moves to vertex $(0, 0, \dots, 0)$ in $H_k$.  Since the cop blocks the only exit from $H_k$, the robber must remain within $H_k$; since the cop can reach any vertex of $H_k$ on her next move, she can capture the robber.  This completes the proof that $\speeds(G) \le t_s$ for all positive integers $s$.

    Finally, we argue that $\speeds(G) \ge t_s$ for all positive integers $s$.  If $s > i_{\ell}$, then clearly $\speeds(G) \ge 1 = t_s$.  Suppose instead that $s \le i_{\ell}$, and fix $k$ such that $i_{k-1} < s \le i_{k}$.  We claim that $H_k$ is a retract of $G$.  Indeed, let $\psi:V(G)\to V(H_k)$ be the identity map for all $v\in V(H_k)$ and $\psi(u)=(0,0,\dots,0)$ for all vertices $u\in V(G)-V(H_k)$; it is easily seen that $\psi$ is a retraction.  By Theorem \ref{thm:retract}, we now have $\speeds(G)\ge\speeds(H_k)=t_{i_k} = t_s$, as desired.
\end{proof}
\end{section}

\begin{section}{Cartesian products}\label{sec:Cartesian}

We next turn our attention to the speed-$(s,s)$ game played on Cartesian products of graphs.  The usual model of Cops and Robbers has been thoroughly studied on Cartesian products, e.g. by To\v{s}i\'{c} \cite{Tos87}, by Maamoun and Meyniel \cite{MM87}, and by Neufeld and Nowakowski \cite{NN98}:

\begin{theorem}[To\v{s}i\'{c} \cite{Tos87}]\label{thm:regular_general_product}
For graphs $G$ and $H$, we have $c(G \cart H) \le c(G) + c(H)$.
\end{theorem}

\begin{theorem}[Maamoun and Meyniel \cite{MM87}]\label{thm:regular_product_of_trees}
If $T_1, \dots, T_n$ are trees, then 
$$c(T_1 \cart \dots \cart T_n) = \ceil{\frac{n+1}{2}}.$$
\end{theorem}

\begin{theorem}[Neufeld and Nowakowski \cite{NN98}]\label{thm:regular_product_of_cycles}
Let $G = H_1 \cart \dots \cart H_n$, where each $H_i$ is a cycle of length at least 4.  Then $c(G) = n+1$.
\end{theorem}

We aim to generalize these results to the speed-$(s,s)$ game.  We begin with Theorem \ref{thm:regular_general_product}.  One might hope that $\speeds(G \cart H) \le \speeds(G) + \speeds(H)$, but this is not the case; in fact, $\speeds(G \cart H)$ cannot be bounded above by any function of $\speeds(G)$ and $\speeds(H)$, as we next show.

Our construction makes use of finite projective planes.  Recall that a \textit{finite projective plane of order $q$} is a collection of \textit{points} and \textit{lines} such that:
\begin{itemize}
\item [\textbf{(1)}] each line is a set consisting of $q+1$ points;
\item [\textbf{(2)}] each point is contained in exactly $q+1$ lines;
\item [\textbf{(3)}] any two distinct lines contain exactly one common point; 
\item [\textbf{(4)}] any two distinct points belong to exactly one common line; and
\item [\textbf{(5)}] there exist four points, of which no line contains more than two.
\end{itemize}
It is well-known that finite projective planes exist for any order $q$ that is a power of a prime.

Given a finite projective plane, define the \textit{incidence graph} of the projective plane to be the graph consisting of one vertex for each point, one vertex for each line, and an edge joining a point $p$ with a line $\ell$ if and only if $p \in \ell$.

\begin{theorem}\label{thm:projective_plane_product}
Fix $q$, suppose that a finite projective plane $\mathcal{P}$ of order $q$ exists, and let $P$ be the incidence graph of $\mathcal{P}$.  Then $\speed{2}(P) = 2$, but $\speed{2}(P \cart P) \ge q+1$.
\end{theorem}
\begin{proof}
It is easily seen that $\speed{2}(P) = 2$.  Since every two lines contain a common point, the distance between any two vertices of $P$ corresponding to lines of $\mathcal{P}$ is at most 2.  Likewise, since any two points lie in a common line, the distance between any two vertices corresponding to points of $\mathcal{P}$ is at most 2.  Hence, if one cop starts on a vertex corresponding to a point and the other starts the game on a vertex corresponding to a line, then they can capture the robber on their first turn, no matter where the robber begins.  On the other hand, the robber can easily evade a single cop on $P$.  Suppose without loss of generality that the cop begins on a vertex $v$ corresponding to a line in $\mathcal{P}$; the robber begins on any vertex corresponding to a point in $\mathcal{P}$ that is not adjacent to $v$.  Such a vertex must be at distance at least 3 from $v$, so the cop cannot capture the robber on her first turn.  Note that for any two distinct vertices $w$ and $x$ in $P$, we have $N_2[w] \not = N_2[x]$; consequently, there must be a vertex within distance 2 of the robber that is not within distance 2 of the cop.  The robber moves to such a vertex; the cop cannot reach the robber on her next turn.  The robber may repeat this strategy indefinitely, so he cannot be captured by a single cop.

Now consider the speed-$(2,2)$ game played on $P \cart P$.  To show that $\speed{2}(P \cart P) \ge q+1$, it suffices to show that a robber can perpetually evade $q$ cops.  Before presenting the robber's strategy on $P\cart P$, we establish a useful property of $P$.  

Suppose that the robber occupies a vertex $v$ of $P$ and that $q$ cops occupy vertices $w_1, \dots, w_q$, all distinct from $v$ (though not necessarily distinct from each other).  We claim that there is some neighbor $u$ of $v$ such that $\textrm{dist}(u,w_i) \ge 2$ for all $i \in \{1, \dots, q\}$.  By symmetry, we may suppose that $v$ represents a point $p$ in $\mathcal{P}$.  Every neighbor of $v$ represents a line containing $p$; there are $q+1$ of these.  Of these, each cop occupying a line can occupy at most one neighbor of $v$ (and cannot be adjacent to any others).  Each cop occupying a point cannot occupy any neighbors of $v$ and can be adjacent to only one, since any two distinct points in $\mathcal{P}$ lie on exactly one common line.  Hence, of the $q+1$ neighbors of $v$, at most $q$ can be within distance 1 of a cop, so at least one neighbor of $v$ lies at distance 2 or greater from every cop, as claimed.

Now return to the speed-$(2,2)$ game played on $P \cart P$; we explain how the robber can evade $q$ cops.  We claim that on each of the robber's turns, provided that he has not already been captured, he can move to a vertex that is distance at least 3 from every cop.  Thereby, the robber can always ensure that he cannot be captured on the ensuing cop turn, so he can evade the cops indefinitely.  (Note that this also ensures that the robber can choose an initial position that will prevent capture on the first cop turn: he simply imagines that he already occupies a vertex that doesn't contain a cop, and then ``moves'' to a vertex that lies at distance 3 or greater from every cop.)

Suppose that it is the robber's turn; denote the robber's position by $(u,v)$, and denote the cops' positions by $(w_1,x_1), \dots, (w_q,x_q)$.  By the argument above, $u$ has a neighbor $u'$ in $G$ such that $\textrm{dist}_P(u,w_i) \ge 2$ for all $i$ such that $w_i \not = u$; likewise, $v$ has a neighbor $v'$ in $P$ such that $\textrm{dist}_P(v,x_i) \ge 2$ for all $i$ such that $x_i \not = v$.  The robber now moves to vertex $(u',v')$, which is clearly at distance 2 from $(u,v)$.  We claim that $(u',v')$ is at distance 3 or greater from $(w_i,x_i)$ for all $i \in \{1, \dots, q\}$.  Note that $\textrm{dist}_{P \cart P}((u',v'),(w_i,x_i)) = \textrm{dist}_P(u',w_i) + \textrm{dist}_P(v',x_i)$.  Since $u' \in N(u)$, if $w_i = u$ we have $\textrm{dist}_P(u',w_i) = 1$; otherwise, by choice of $u'$, we have $\textrm{dist}_P(u',w_i) \ge 2$.  Likewise, if $x_i = v$ we have $\textrm{dist}_P(v',x_i) = 1$, and otherwise $\textrm{dist}_P(v',x_i) \ge 2$.  Since the robber had not already been captured we cannot have both $u=w_i$ and $v=x_i$ for the same value of $i$, hence $$\textrm{dist}_{P \cart P}((u',v'),(w_i,x_i)) = \textrm{dist}_P(u',w_i) + \textrm{dist}_P(v',x_i) \ge 1+2 = 3.$$
Thus the robber has moved to a vertex at distance at least 3 from every cop, which completes the proof.
\end{proof}

Since there exist projective planes of arbitrarily large order, Theorem~\ref{thm:projective_plane_product} shows that $\speed{2}(G \cart H)$ can be arbitrarily large even when $\speed{2}(G)+\speed{2}(H) = 4$.  We remark that the strategy employed by the robber in the proof Theorem \ref{thm:projective_plane_product} is, in some sense, to play two speed-1 games simultaneously, one on each factor of $P \cart P$.  This strategy works well because even though $\speed{2}(P) = 2$, it was shown in \cite{Pra10} that $c(P) = q+1$.  It would be interesting to determine whether one can, in general, bound $\speed{s}(G \cart H)$ above by some function of $\speed{t}(G)$ and $\speed{s-t}(G)$ for arbitrary $t \in \{1, \dots, s-1\}$.\\

 We next consider the speed-$(s,s)$ game played on Cartesian products of trees.  We begin with some terminology.  Suppose $G = G_1 \cart \dots \cart G_d$ for some graphs $G_1, \dots, G_d$.  As is standard, we view the vertex set of $G$ as the set $\{(x_1, \dots, x_d) \, : \, x_i \in V(G_i) \text{ for all } i\}$; for a vertex $v = (x_1, \dots, x_d)$ we call $x_i$ the \textit{$i$th coordinate} of $v$.  For any $k \in \{1, \dots, d\}$, the \textit{projection} of $G$ onto $G_k$ is the retraction $\phi$ defined by $\phi((x_1, \dots, x_d)) = x_k$; we sometimes refer to the image of this map as \textit{dimension $k$} or \textit{the $k$th dimension} of $G$.  For two vertices $x=(x_1,\dots,x_d)$ and $y=(y_1,\dots,y_d)$ in $G$ and a given $k\in\{1,\dots,d\}$, we define the \textit{distance in dimension $k$} from $x$ to $y$ as $\dist_{G_k}(x_k, y_k)$.  Likewise, when we refer to the distance between a cop and the robber in a given dimension $k$, we mean the distance in dimension $k$ from the vertex that the cop occupies to the vertex that the robber occupies.  When a player occupies vertex $(x_1, \dots, x_d)$, we refer to $x_k$ as their \textit{shadow} in dimension $k$. We say that a cop has \textit{captured the robber's shadow} when the cop and robber have the same shadow in dimension $k$.

The following simple observation, in conjunction with Theorem \ref{thm:retract}, will be very helpful in establishing lower bounds on $\speeds$.

\begin{proposition}\label{prop:cartesian_retract}
If $G'$ is a retract of $G$ and $H'$ is a retract of $H$, then $G'\cart H'$ is a retract of $G\cart H$.
\end{proposition}

\begin{proof}
Let $\phi_1$ be a retraction from $G$ to $G'$ and let $\phi_2$ be a retraction from $H$ to $H'$.  We will show that the map $\phi:V(G\cart H)\to V(G'\cart H')$ given by $\phi(u,v)=\left(\phi_1(u),\phi_2(v)\right)$ is a retraction.

Since $\phi_1$ and $\phi_2$ are retractions, for any $(x,y)\in V(G'\cart H')$ we have $\phi(x,y)=\left(\phi_1(x),\phi_2(y)\right)=(x,y)$.

It remains to show that $\phi$ is a homomorphism.  Let $(u_1,u_2)$ and $(v_1,v_2)$ be two adjacent vertices in $G\cart H$; we must show that $\phi((u_1,u_2))\phi((v_1,v_2)) \in E(G' \cart H')$.  By definition of the Cartesian product, either $u_1=v_1$ and $u_2v_2\in E(H)$ or $u_2=v_2$ and $u_1v_1\in E(G)$; by symmetry we may assume the former.  Now
\[\phi((u_1,u_2))\phi((v_1,v_2)) = \left(\phi_1(u_1),\phi_2(u_2)\right)\left(\phi_1(v_1),\phi_2(v_2)\right).\]
By assumption $u_1 = v_1$ hence $\phi_1(u_1) = \phi_1(v_1)$, and because $\phi_2$ is a homomorphism we have $\phi_2(u_2)\phi_2(v_2)\in E(H')$, so it follows that $\phi((u_1,u_2))\phi((v_1,v_2))\in E(G'\cart H')$.
%By definition of $\phi$, we have $\phi\left((u_1,u_2)\right)=\left(\phi_1(u_1),\phi_2(u_2)\right)$ and $\phi\left((v_1,v_2)\right)=\left(\phi_1(v_1),\phi_2(v_2)\right)$.  Since $\phi_1$ and $\pd since $H'$ \textcolor{blue}{(Do I mean $G'$ here, not $H'$?)} is assumed to be reflexive, we have $\phi_1(u_1)=\phi_1(v_1)$ and $\phi_2(u_2)\phi_2(v_2)\in E(H')$, so $\phi\left((u_1,u_2)\right)\phi\left((v_1,v_2)\right)\in E(G'\cart H')$.  

It follows that $\phi$ is a retraction from $G \cart H$ to $G' \cart H'$, as claimed.
\end{proof}

%We begin by considering the speed-$s$ game on products of trees.  
We can obtain an easy upper bound on $\speed{s}$ as a corollary of Theorem~\ref{thm:speed_s_active} and a result of Neufeld and Nowakowski (\cite{NN98}, Theorem 2.5):
\begin{corollary}\label{cor:cartesian_product_trees_upper}
Fix $d \ge 3$ and let $T_1, \dots, T_d$ be trees.  For $s \ge 2$, we have
$$\speed{s}(T_1 \cart \dots \cart T_d) \le \ceil{\frac{d}{2}}.$$
\end{corollary}
\begin{proof}
Let $G = T_1 \cart \dots \cart T_d$.  By Theorem~\ref{thm:speed_s_active} we have $\speed{s}(G) \le c'(G)$; by Theorem 2.5 in \cite{NN98}, we have $c'(G) \le \ceil{\frac{d}{2}}$.
\end{proof}
Note that Corollary \ref{cor:cartesian_product_trees_upper} requires $d \ge 3$; we handle the cases $d=1$ and $d=2$ later, in Theorem \ref{thm:speed_s_large_grids}. 

Using arguments inspired by those in \cite{NN98}, we can obtain more refined upper bounds on $\speeds$.  The following (fairly technical) lemma will be helpful in establishing these bounds.  The proof was inspired by viewing $G$ and $H$ each as a Cartesian product of paths.  It may behoove the reader to consider this case when trying to digest the entire proof.  However, we state and prove the lemma in full generality, with as few restrictions on $G$ and $H$ as possible.

\begin{lemma}\label{lem:active_and_change}
Fix $s \ge 2$, let $H$ be a graph, and suppose that $k$ cops can capture a robber in the speed-$(s,s)$ game on $H$ under the additional restriction that on each turn, the robber must move to a vertex whose distance from his current position is either $s-1$ or $s$.  Then for any graph $G$,
$$\speed{s}(G \cart H) \le c^-(G) + k.$$
\end{lemma}
\begin{proof}
Let $m = c^-(G)$; we give a strategy for $k+m$ cops to play on $G \cart H$.  Label the cops $C_1,\dots,C_{k+m}$.  Each vertex of $G \cart H$ can be viewed as an ordered pair $(u,v)$ with $u \in V(G)$ and $v \in V(H)$; we refer to $u$ as the \textit{$G$-coordinate} and to $v$ as the \textit{$H$-coordinate} of this vertex.  We refer to cops $C_1, \dots, C_m$ as \textit{$G$-cops} and to cops $C_{m+1}, \dots, C_{m+k}$ as \textit{$H$-cops}.  Each cop will execute her strategy in two \textit{phases}.  For a given $G$-cop (resp. $H$-cop), \textit{phase one} consists of moving to capture the robber's shadow in the $H$-coordinate (resp. $G$-coordinate);  once she has accomplished this, she will transition to \textit{phase two}, wherein she will attempt to capture the robber's shadow the $G$-coordinate (resp. $H$-coordinate).  Loosely, the $G$-cops attempt to capture the robber in $H$ first, then $G$; the $H$-cops attempt to capture the robber in $G$ first, then $H$. 
          
We now present phase one of the cops' strategy.  Fix a cop $C$, and suppose $C$ is a $G$-cop (the case where $C$ is an $H$-cop is symmetric).  In this phase, $C$ aims to capture the robber's shadow in the $H$-coordinate.  Let $D$ denote the distance, in $H$, between the $H$-coordinate of $C$'s position and the $H$-coordinate of the robber's position.  Note that if at any point $D = 0$, then cop $C$ has completed phase one and may progress to phase two.  (We argue later that once she has entered phase two, she will play so as to ensure that $D=0$ after every cop turn, and hence she will remain in phase two for the rest of the game.)  We show how cop $C$ can play to ensure that $D$ will not increase in the course of a robber turn and ensuing cop turn.

To see this, consider the state of the game at the beginning of a robber turn.  The robber takes at most $s$ steps, hence he can increase $D$ by at most $s$.  On the subsequent cop turn, $C$ takes one step toward the robber in the $H$-coordinate and repeats until $D=0$ or until she has taken $s$ steps, whichever occurs first.  In the former case, the cop has successfully completed phase one.  In the latter case, she has decreased $D$ by $s$ on her turn, so across both players' turns, $D$ cannot have increased.  Moreover, note that if the robber takes fewer than $s$ steps away from $C$ within the $H$-coordinate, then in fact $D$ must decrease.  Whenever $D$ decreases, we say that $C$ \textit{makes progress} toward completing phase one.  Note that, clearly, each cop can only make progress toward completing phase one a bounded number of times before reaching phase two.
            
Next we consider phase two.  We start by explaining the strategy for the $G$-cops; the $H$-cops will play slightly differently, and we will return to them later.  If not all of the $G$-cops have yet reached phase two, then each $G$-cop who \textit{has} reached phase two simply mimics the robber's movements in the $H$-coordinate, thereby ensuring that they remain in phase two.  Once all of the $G$-cops have reached phase two, they play as follows. Suppose that the robber moves from vertex $(u_1, v_1)$ to vertex $(u_2, v_2)$.  We consider three cases:
\begin{itemize}
\item \textbf{Case 1:} $u_1 \not = u_2$.  In this case, each $G$-cop first mimics the robber's movements in the $H$-coordinate so as to remain in phase two.  Next, the $G$ cops respond to each of the robber's steps in the $G$-coordinate one-by-one, using a winning strategy for the semi-active game on $G$ (as in the proof of Theorem \ref{thm:speed_s_active}).\\

\item \textbf{Case 2:} $u_1 = u_2$ and $\textrm{dist}_H(v_1, v_2) \le s-2$.  In this case, the $G$-cops take $\textrm{dist}_H(v_1, v_2)$ steps toward the robber in the $H$-coordinate, once again ensuring that they remain in phase two.  They are still permitted to travel at least another two steps.  As in the proof of Theorem \ref{thm:speed_s_active}, the $G$-cops imagine that the robber has moved, in the $G$-coordinate, from $u_1$ to some neighbor $u'$ and then back to $u_1$, then respond to each of these steps using a winning strategy for the semi-active game on $G$.\\

\item \textbf{Case 3:} $u_1 = u_2$ and $\textrm{dist}_H(v_1,v_2) \in \{s-1, s\}$.  In this case, the $G$-cops simply mirror the robber's movements in the $H$-coordinate and do not move at all in the $G$-coordinate.
\end{itemize}
If either Case 1 or Case 2 applies, then we say that the $G$-cops \textit{make progress} toward completing phase two.  Note that since the cops are following a winning strategy for the semi-active game on $G$, they can only make progress toward completing phase two a bounded number of times before actually capturing the robber.

Finally, we explain how the $H$-cops play in phase two.  As before, if not all of the $H$-cops have reached phase two, then those $H$-cops who are already in phase two simply mimic the robber's movements in the $G$-coordinate, thereby ensuring that they remain in phase two.  Once all of the $H$-cops have reached phase two, they again mimic the robber's movements in the $G$-coordinate, but now they may choose to move in the $H$-coordinate as well.  In particular, suppose the robber moves from vertex $(u_1, v_1)$ to $(u_2, v_2)$.  If $u_1 = u_2$ and $\textrm{dist}_H(v_1,v_2) \in \{s-1, s\}$, then the $H$-cops follow, in the $H$-coordinate, a winning strategy for the speed-$(s,s)$ game on $H$ under the restriction that the robber must always move to a vertex distance $s-1$ or $s$ away from his previous position.  Whenever this happens, we say that the $H$-cops \textit{make progress} toward completing phase two.  Note that if the $H$-cops make progress toward completing phase two in several consecutive rounds, but then fail to make progress on the next round (because $u_1 \not = u_2$ or because $\textrm{dist}_H(v_1,v_2) < s-1$), then they may be forced to deviate from their winning strategy on $H$.  Nevertheless, by choice of $k$, there exists some $N$ such that if the $H$-cops make progress toward completing phase two on $N$ consecutive rounds, then some $H$-cop will capture the robber.

Finally, we argue that the cops' strategy is guaranteed to capture the robber.  When the robber moves distance $s-2$ or less in the $H$-coordinate, each $G$-cop in phase one makes progress toward completing phase one or, if all of the $G$-cops are already in phase two, they make progress toward completing phase two.  However, whenever the robber moves distance $s-1$ or more in the $H$-coordinate, each $H$-cop in phase one makes progress toward completing phase one or, if all of the $H$-cops are already in phase two, they make progress toward completing phase two.  It follows that either some $G$-cop eventually captures the robber, or there will be enough consecutive rounds in which the robber moves distance $s-1$ or more in the $H$-coordinate that some $H$-cop will capture the robber.
\end{proof}

Before jumping into general results on Cartesian products of trees, we provide another useful lemma, this time for robber strategies on these graphs.  We say that a robber takes a step \textit{toward} a cop $C$ on the graph $G$ if his step decreases the distance from $C$ to the robber.  Conversely, the robber takes a step \textit{away} from $C$ if his step increases the distance from $C$ to the robber.  Roughly speaking, our next lemma states that the robber can avoid being captured by a given cop if, on each robber turn, he covers less than half the distance between himself and the cop.

\begin{lemma}\label{lem:distance_x/2}
    Fix positive integers $s$ and $d$ with $s\ge 2$, let $T_1,\dots T_d$ be trees, and let $G = T_1\cart\dots\cart T_d$.  Let the robber and any number of cops play the speed-$(s,s)$ game on $G$.  If the robber has not yet been caught and takes $s$ steps on his turn, then a cop $C$ at distance $x$ from the robber cannot catch the robber on the ensuing turn if the robber moves fewer than $x/2$ steps toward $C$.
\end{lemma}
\begin{proof}
    Suppose a cop $C$ is currently at distance $x$ from the robber.  
    %We claim that if the robber moves distance $s$ from his current position, then $C$ cannot capture him on the ensuing cop turn provided that the robber has taken fewer than $x/2$ steps toward $C$.  To see this, note that if 
    If the robber takes $m$ steps toward $C$ and $s-m$ steps away, then the distance between $C$ and the robber after the robber's move is $x+(s-m)-m$, which is strictly greater than $s$ provided $m<x/2$.
\end{proof}

We are finally ready to analyze the speed-$(s,s)$ game on Cartesian products of trees.  

\begin{theorem}\label{thm:speed_s_large_grids}
Fix positive integers $s$ and $d$ with $s \ge 2$, let $T_1, \dots, T_d$ be trees, and let $G = T_1 \cart \dots \cart T_d$.  
\begin{enumerate}[label={(\alph*)}]
\item If $d \in \{1, 2\}$, then $\speeds(G) = 1$.
\item If $\diam(T_i) \ge 2s$ for all $i\in\{1,\dots,d\}$, then $\speeds(G) = \ceil{d/2}$.
% \item If $n_i \ge 4s$ for all $i \in \{1, \dots, d\}$, then $\speeds(H) = d+1$.
\end{enumerate}
\end{theorem}
\begin{proof}
    \begin{enumerate}[label={\textbf{(\alph*)}}]
    
    \item \noindent If $d=1$, then $G$ is a tree, and it is obvious that a single cop suffices to capture the robber.\\ 

    Suppose instead that $d=2$.  We show how a single cop can capture the robber on $G$.  The cop begins on an arbitrary vertex of $G$.  For $i\in\{1,2\}$, let $D(i)$ be the distance from the cop to the robber in dimension $i$.  On the cop's turn, she will take $s$ steps in sequence, one step at a time (unless she can capture the robber with fewer than $s$ steps).  On each step, she moves to minimize $\displaystyle\max_{i\in\{1,2\}}\{D(i)\}$ without making $D(1)$ or $D(2)$ equal to zero.  That is, if the cop is currently at distance 1 from the robber in each dimension but can only take one more step on her current turn, then she ends her turn.  (It might happen that the robber moves so as to make $D(1)$ or $D(2)$ equal to zero at some point, but the cop will never move to do so.)
    
    We claim that the total distance from the cop to the robber, i.e. $D(1)+D(2)$, cannot decrease from one round to the next and, moreover, that it must eventually decrease until $D(1) \le 1$ and $D(2) \le 1$ after some cop turn.  To see this, note that each step the cop takes decreases $D(1)+D(2)$.  Consequently, on the cop's turn, she either captures the robber, realizes $D(1) \le 1$ and $D(2) \le 1$, or decreases the total distance by exactly $s$.  Since the robber takes at most $s$ steps away from the cop on each turn, at the end of the round, the robber's distance from the cop is no larger than it was at the beginning of the previous round.  
    
    If in some round the robber does not move $s$ steps away from the cop, then in the course of that round the total distance from the cop to the robber will have decreased.  Thus, in order to ensure the total distance does not decrease, the robber must move $s$ steps away from the cop on each of his turns.  Note that the robber's shadow in each dimension cannot move through the cop's shadow, or else the robber would be moving at least one step toward the cop on his turn, and thus $D(1)+D(2)$ would decrease.  Furthermore, since each factor of the graph is a tree, and since the cop never moves so as to make $D(1)$ or $D(2)$ equal to zero (and thereby avoids moving onto the robber's shadow), the robber cannot move $s$ steps away from the cop indefinitely: eventually, his shadow in some dimension must move toward the cop's shadow, hence the robber must move fewer than $s$ steps away from the cop.  Thus, $D(1)+D(2)$ must eventually decrease, so the cop eventually realizes $D(1) \le 1$ and $D(2)\le 1$ at the end of some cop turn.

    The cop now seeks to realize $D(1) = D(2) = 1$.  If this is not already the case, one of $D(1)$ and $D(2)$ must be 1 and the other 0; without loss of generality, suppose $D(1) = 1$ and $D(2) = 0$.  Henceforth, if on any turn the robber ever takes fewer than $s$ steps away from the cop, then she may move so as to make $D(1)=D(2)=1$ on her next turn.  Likewise, if the robber ever moves in dimension 2, then the cop may again force $D(1)=D(2)=1$ on her next turn.  Thus the robber must take $s$ steps away from the cop in dimension 1 on each robber turn; since $T_1$ is a tree, the robber cannot keep this up indefinitely.  Hence the cop can eventually force $D(1)=D(2)=1$.

    From this point forward, if on any robber turn the robber takes even a single step toward the cop, then she will be able to capture him on her ensuing turn.  Likewise, if the robber takes $s-2$ steps or fewer on any turn, then the cop can capture him.  Hence, the robber must take at least $s-1$ steps away from the cop on every turn.  However, as argued above, he cannot keep this up forever: on some robber turn he will be forced to take fewer than $s-1$ steps or take at least one step toward the cop, so the cop can capture him on the ensuing cop turn.\\
    
    \item Suppose $d \ge 3$ and $\diam(T_i) \ge 2s$ for all $i$ in $\{1, \dots, k\}$.  For the upper bound, Corollary \ref{cor:cartesian_product_trees_upper} states that $\speeds(G) \le \ceil{d/2}$.  
    
    For the lower bound, we first claim that it suffices to consider the case where $d$ is odd.  Indeed, if $d = 2k$ for some integer $k$, then Theorem \ref{thm:retract} shows that 
    \[\speeds(G) = \speeds(T_1 \cart \dots \cart T_{2k}) \ge \speeds(T_1 \cart \dots \cart T_{2k-1} \cart K_1) = \speeds(T_1 \cart \dots T_{2k-1}),\]
    since $K_1$ is trivially a retract of $T_{2k}$ and $T_1 \cart \dots \cart T_{2k-1}\cong T_1 \cart \dots \cart T_{2k-1}\cart K_1$; thus, to show that $\speeds(G) \ge \ceil{d/2} = k$, it suffices to show that $\speeds(T_1 \cart \dots T_{2k-1}) \ge \ceil{(2k-1)/2} = k$.  
    
    To further simplify the situation, let $n_i = \diam(T_i)$ for all $i \in \{1, \dots, d\}$ and note that $P_{n_i}$ is a retract of $T_i$; so, again applying Theorem \ref{thm:retract}, we have
    $$\speeds(G) = \speeds(T_1 \cart \dots \cart T_d) \ge \speeds(P_{n_1} \cart \dots \cart P_{n_d}).$$
    Thus motivated, define $G' = P_{n_1} \cart \dots \cart P_{n_d}$, and suppose $d=2k+1$ for some positive integer $k$; we show how the robber can evade $k$ cops on $G'$.  Loosely, the robber's strategy will be to choose a dimension in which there are no cops ``close'' to his current position and to move $s$ steps in that dimension.  
    
%    Before explaining how the robber chooses a dimension in which to move, we first make an observation about how the robber must move to avoid capture on the cops' subsequent turn.  In particular, suppose a cop $C$ is currently at distance $x$ from the robber.  We claim that if the robber moves distance $s$ from his current position, then $C$ cannot capture him on the ensuing cop turn provided that the robber has taken fewer than $x/2$ steps toward $C$.  To see this, note that if the robber takes $m$ steps toward $C$ and $s-m$ steps away, then the distance between $C$ and the robber after the robber's move is $x+(s-m)-m$, which is strictly greater than $s$ provided that $m<x/2$.  \textcolor{blue}{I vote that the preceding paragraph be a lemma or a proposition or a fact that lies outside of any proof, as we will be using this exact idea in the proofs of the lower bounds for hypercubes.}
    
    The following definitions will be helpful in describing the robber's strategy in more detail.  When a player occupies vertex $(y_1,y_2,\dots,y_d)$, we say that they move \textit{forward} (resp. \textit{backward}) in dimension $i$ if they move to increase (resp. decrease) $y_i$.  If the robber occupies the vertex $(y_1,\dots,y_d)$ and a cop occupies the vertex $(z_1,\dots,z_d)$, we say that the cop lies \textit{in front of} (resp. \textit{behind}) the robber in dimension $i$ if $z_i>y_i$ (resp. $z_i<y_i$).  We say that a cop $C$ at distance $x$ from the robber \textit{blocks} the robber's forward (resp. backward) direction in dimension $i$ if the distance between $C$ and the robber in dimension $i$ is at least $x/2$ and the cop lies ahead of (resp. behind) the robber in dimension $i$.  Each cop can block at most two directions, so in total, the $k$ cops can block at most $2k$ of the robber's directions across all dimensions.  Since $d = 2k+1$, there must be at least one dimension $i$ in which neither direction is blocked; since $n_i \ge 2s$, the robber can move $s$ steps in at least one of the two directions of that dimension.  As a result of Lemma \ref{lem:distance_x/2}, this prevents any of the cops from capturing the robber on the ensuing cop turn.  By repeating this strategy, the robber can perpetually evade all of the cops.
    \end{enumerate}

\end{proof}

The robber strategy in Theorem \ref{thm:speed_s_large_grids}(b) necessitates that each factor has diameter at least $2s$, so that the robber will always have enough room to move $s$ steps in an unblocked direction.  What happens if the factors are small relative to $s$?  In particular, what happens if $G$ is the $d$-dimensional hypercube, $Q_d$?  Recall that $Q_d$ represents the $d$-fold Cartesian product of $K_2$.

Lemma \ref{lem:active_and_change} easily yields an upper bound on $\speeds(Q_d)$.  (We remark that in fact this argument can be applied to $T_1 \cart \dots \cart T_d$ whenever at least $2s-1$ of the $T_i$ are $K_2$.)

\begin{theorem}\label{thm:hypercube_upper}
For all positive integers $s \ge 2$ and $d \ge 2$ with $d \ge 2s$, we have 
\[\speeds(Q_d) \le \ceil{\frac{d-2s+3}{2}}.\]
\end{theorem}
\begin{proof}
By Theorem \ref{thm:retract}, since $Q_d$ is a retract of $Q_{d+1}$ for all $d$, we have $\speeds(Q_d) \le \speeds(Q_{d+1})$; consequently, it suffices to consider the case where $d-2s+3$ is even.  Let $k = \frac{d-2s+3}{2}$, so that $d = 2k+2s-3$ for some integer $k \ge 2$; we must show that $\speeds(Q_d) \le k$.  Note that $Q_d = Q_{2k+2s-3} = Q_{2k-2} \cart Q_{2s-1}$.  By Theorem 2.5 in \cite{NN98}, we have $c^-(Q_{2k-2}) \le c'(Q_{2k-2}) \le \ceil{(2k-2)/2} = k-1$.  Now consider the speed-$(s,s)$ game on $Q_{2s-1}$ under the additional restriction that on each robber turn, the robber must move to a vertex either distance $s-1$ or distance $s$ from his current position; we claim that a single cop can win.  On the first cop turn, if the robber is within distance $s$, then the cop can capture him.  If instead the robber is distance $s+1$ or farther from the cop, then the cop can move to the unique vertex at distance $2s-1$ from the robber.  On the ensuing robber turn, the robber must move to a vertex at distance $s-1$ or greater from his current position, which will put him at distance $s$ or less from the cop, so the cop can win on her second turn.  In either case, the cop wins.

Now, Lemma \ref{lem:active_and_change} yields 
\[\speeds(Q_{2k+2s-3}) \le c^-(Q_{2k-2}) + 1 \le k-1+1 = k,\]
as claimed.
\end{proof}

Proving lower bounds on $\speeds(Q_d)$ is less straightforward.  When $s=2$, we can obtain a lower bound that matches the upper bound given by Theorem \ref{thm:hypercube_upper}; we present that next.  (Afterward, we will consider speeds higher than 2.)

\begin{theorem}\label{thm:speed_2}
For any positive integer $d$ and nontrivial trees $T_1, \dots, T_d$, we have 
\[\speedtwo(T_1 \cart \dots \cart T_d) \ge \floor{d/2}.\]
\end{theorem}
\begin{proof}
Since each $T_i$ is nontrivial, it follows that $K_2$ is a retract of $T_i$ for all $i \in \{1, \dots, k\}$, hence (by Proposition \ref{prop:cartesian_retract}), $Q_d$ is a retract of $T_1 \cart \dots \cart T_d$.  Consequently, by Theorem \ref{thm:retract}, we have $\speedtwo(T_1 \cart \dots \cart T_d) \ge \speedtwo(Q_d)$.  It thus suffices to prove that $\speedtwo(Q_d) \ge \floor{d/2}$.  Moreover, since $Q_{2k}$ is a retract of $Q_{2k+1}$, we have $\speedtwo(Q_{2k+1}) \ge \speedtwo(Q_{2k})$, so it suffices to show that a robber can evade $k-1$ cops on $Q_{2k}$.  Note that when a player takes one step in the hypercube, they change one coordinate of their current position from 0 to 1 or from 1 to 0.  Thus, when describing a player's moves, it suffices to list the dimensions in which they move.

To evade capture, the robber simply needs to ensure that after every robber move, he ends up at a vertex at distance three or more from every cop; we refer to such a vertex as a \textit{safe} vertex.  We first claim that after the cops choose their starting vertices, the robber can find a safe vertex at which to begin the game.  To see this, note that for any vertex $v$ in $Q_{2k}$, we have 
\[\size{N_2[v]} = \binom{2k}{2} + \binom{2k}{1} + \binom{2k}{0} = \frac{4k^2+2k+2}{2}.\] 
Consequently, the number of vertices within distance two of some cop is at most $(k-1)(4k^2+2k+2)/2$.  In total, $Q_{2k}$ has $2^{2k}$ vertices, which is strictly larger than $(k-1)(4k^2+2k+2)/2$ for all positive integers $k$.  Thus there must be at least one safe vertex; the robber starts at any such vertex.   
        
We next show that after the cops move, the robber can move to another safe vertex.  First note that the cops cannot capture the robber on their turn, since all cops begin the turn at distance three or more from the robber.  After the cops move, the robber must decide which dimensions to move in.  To assist him in this endeavor, he assigns weights to the dimensions of $Q_{2k}$ in the following way.  Initially, all dimensions receive weight 0.  For each cop at distance one from the robber, the robber adds weight 1 to the single dimension where the positions of the cop and the robber differ.  Similarly, for each cop at distance two from the robber, the robber adds weight 1 to both of the dimensions where the positions of the cop and the robber differ.  For each cop at distance three or distance four from the robber, the robber adds weight $1/2$ to each dimension where the positions of the cop and the robber differ.  
        
The robber uses these weights to determine his next move.  First note that cops at distance five or more do not pose a threat: even if the robber takes two steps toward such a cop, he will still be at distance three or more from that cop.  Thus, henceforth we concern ourselves only with cops within distance four of the robber.  If a dimension has weight 0, then moving in this dimension increases the robber's distance from every cop.  Consequently, if there are at least two dimensions with weight 0, then the robber may move in both of these dimensions to increase his distance to each cop by two, which ensures that he reaches a safe vertex.  If a dimension has weight 1/2, then by moving in that dimension, the robber moves one step toward a cop that was at distance three or four, and he moves away from every other cop.  Thus, if the robber moves in a dimension of weight 0 and a dimension of weight 1/2, then he again reaches a safe vertex.  (Note that even though he might take one step closer to a cop at distance three when moving in the dimension with weight 1/2, he would subsequently take one step away from the same cop when moving in the dimension with weight 0.)  Finally, if there are two dimensions with weight 1/2 that received this weight from two different cops, then the robber may move in both of these dimensions and end up at a safe vertex. 
        
We claim that the robber can always find two dimensions of the desired form.  Since there are $k-1$ cops, and each cop can induce a total weight of 2 across all dimensions, the total weight assigned to all $2k$ dimensions is at most $2k-2$.  Hence, at least two dimensions must have weight strictly less than 1.  If any dimension has weight 0, then either the robber can find two dimensions of weight 0 or one dimension of weight 0 and one of weight 1/2; in either case he can move to a safe vertex.  Suppose instead that no dimension has weight 0.  Since every vertex has weight at least 1/2, and the total weight across all dimensions is at most $2k-2$, there must be at least four dimensions with weight 1/2.  If in fact there are five such dimensions, then the robber can find two dimensions with weight 1/2 that received their weight from two different cops, in which case he can move to a safe vertex as explained above.  Otherwise, the total weight must be exactly $2k-2$, and there must be $2k-4$ dimensions of weight 1 and four dimensions of weight 1/2 (all four of which necessarily received their weight from the same cop).  This can happen only if there are $k-1$ cops at distance two from the robber and one cop at distance four.  In this case, the robber can take one step in any dimension of weight 1/2, which will leave him at distance three from every cop.

By repeating this strategy, the robber can evade the cops indefinitely; hence, $\speedtwo(Q_{2k}) \ge k$.
\end{proof}

Theorems \ref{thm:hypercube_upper} and \ref{thm:speed_2} yield the following corollary.

\begin{corollary}\label{cor:speed_2_hypercube}
\[\speedtwo(Q_d) = \left \{\begin{array}{ll}2, \quad \quad &\text{if }d=3\\\floor{\frac{d}{2}}, &\text{otherwise}\end{array}\right . .\]
\end{corollary}
\begin{proof}
When $d \not = 3$, the lower bound is a special case of Theorem \ref{thm:speed_2}; the upper bound is clear for $d \le 2$ and is a special case of Theorem \ref{thm:hypercube_upper} when $d \ge 4$.

When $d=3$, we must show that one cop cannot capture a robber in the speed-$(2,2)$ game on $Q_3$, but two cops can.  To avoid a single cop, the robber can always move to occupy the unique vertex at distance 3 from the cop.  Conversely, two cops can capture the robber by beginning the game on vertices at distance 3 from each other; no matter where the robber starts, one of the cops will be able to capture him on the first cop turn.
\end{proof}

We remark that when $s=2$, $d \ge 4$, and at least three of the $T_i$ are $K_2$, the lower bound in Theorem \ref{thm:speed_2} matches the upper bound obtained from Lemma \ref{lem:active_and_change} as in the proof of Theorem \ref{thm:hypercube_upper}; when each $T_i$ has diameter at least 4 and $d$ is even, it falls just short of the upper bound from Corollary \ref{cor:cartesian_product_trees_upper}.  %We discuss this detail further in Section \ref{sec:open_questions}, as it will spawn an open problem.
It would be interesting to determine the choices of $T_i$ that make $\speedtwo(G)$ equal to $\floor{d/2}$, and those that make it equal to $\ceil{d/2}$.

Unfortunately, the ideas from the proof of Theorem \ref{thm:speed_2} do not generalize to higher speeds.  We believe that the upper bound in Theorem \ref{thm:hypercube_upper} is tight; proving a matching lower bound would be equivalent to showing that if $s\ge 3$ and $d\ge 2k+2s-2$, then the robber can able to evade $k$ cops on $Q_d$.  Proving this in general seems surprisingly difficult.  With considerable effort, one can show that the robber can evade $k$ cops on $Q_d$ when $d$ is ``close'' to $2k+2s-2$, provided that $s \gg k$; we omit the (rather lengthy) proof and refer the interested reader to \cite{TowThesis}, Theorem 3.3.13.

\begin{theorem} \label{thm:speed_s_hypercube_lower_bound}
Fix positive integers $s$ and $k$.  If $d > 2s+2k+k\lg(2s+1)$, then the robber can perpetually evade $k$ cops in the speed-$(s,s)$ game on $Q_d$.
\end{theorem}
% \begin{proof}
%     See Theorem 3.3.13 in \cite{TowThesis}.
% \end{proof}

% Do we have a simple argument for the lower bound?  We probably should...
% 
Finally, we briefly discuss the speed-$(s,s)$ game played on the Cartesian product of cycles.  Let $G = C_{n_1} \cart C_{n_2} \cart \dots \cart C_{n_k}$.  If $n_i \ge 4$ for all $i$, then Theorem \ref{thm:regular_product_of_cycles} states that $c(G) = k+1$; hence for all $s$, we have $\speeds(G) \le c(G) = k+1$.  On the other hand, an argument very similar to that used for the lower bound in Theorem \ref{thm:speed_s_large_grids}(b) shows that if $n_i \ge 4s+1$ for all $i$, then $\speeds(G) \ge k$.  Thus, provided that all of the factor cycles are ``large'', we have $\speeds(G) \in \{k, k+1\}$.  Determining the exact value of $\speeds(G)$ seems difficult in general; however, our next result resolves a special case.  We begin with a lemma.

\begin{lemma}\label{lem:rel_prime}
If $n_1$ and $n_2$ are relatively prime, then $c^-(C_{n_1} \cart C_{n_2}) = 2$.
\end{lemma}
\begin{proof}
Let $G = C_{n_1} \cart C_{n_2}$.

For the lower bound, if $\max\{n_1, n_2\} \ge 5$, then it is clear that $c^-(G) > 1$, since one cop cannot capture a robber in the semi-active game on $C_n$ when $n \ge 5$.  For the cases where $3 \le n_1, n_2 \le 4$, it is clear through inspection that $c^-(G) > 1$.

For the upper bound, we assign coordinates to the vertices of $G$: for all integers $i,i',j,j'$, we say that vertices $(i,j)$ and $(i',j')$ are adjacent provided that $i=i'$ and $\size{j-j'} = 1$ or if $\size{i-i'} = 1$ and $j=j'$.  We say that vertex $(i,j)$ is \textit{to the left of} $(i+1,j)$, \textit{to the right of} $(i-1,j)$, \textit{above} $(i,j-1)$, and \textit{below} $(i,j+1)$.  Note that we do not restrict the range of $i, i', j$, and $j'$; in particular, for all integers $i,j,c$, and $d$, vertex $(i,j)$ is the same as vertex $(i+cn_1, j+dn_2)$.

We show how two cops can capture a robber on $G$.  The cops initially position themselves on any vertex of $G$.  After the robber chooses his initial position, by adjusting coordinates, we may suppose by symmetry that the robber occupies vertex $(0,0)$.  We claim that if, for some positive integers $a$ and $b$, the cops can occupy vertices $(-a,-a)$ and $(b,b)$ at the end of some cop turn, then they can capture the robber.  Suppose cop $C_1$ occupies $(-a,-a)$ and cop $C_2$ occupies $(b,b)$.  Whenever the robber moves left (resp. down), $C_1$ moves up (resp. right) and $C_2$ moves left (resp. down); after adjusting coordinates so that the robber again occupies $(0,0)$, cop $C_1$ now occupies $(-(a-1), -(a-1))$, while $C_2$ occupies $(b,b)$.  Similarly, when the robber moves right (resp. up), $C_1$ moves right (resp. up) and $C_2$ moves down (resp. left); now, after adjusting coordinates so that the robber occupies $(0,0)$, cop $C_1$ occupies $(-a,-a)$ and $C_2$ occupies $(b-1,b-1)$.  Thus, no matter how the robber moves, in each round of the game one cop gets closer to the robber while the other remains the same distance away.  Consequently, some cop must eventually capture the robber.

Finally, we argue that the cops can reach such a position.  Suppose $C_1$ currently occupies vertex $(i,j)$ and $C_2$ occupies $(i',j')$; by symmetry we may suppose that $i, j \le 0$ and that $i',j' \ge 0$.  Since $n_1$ and $n_2$ are relatively prime, there exist positive integers $c$ and $d$ such that $i-cn_1 = j-dn_2$.  Hence, $C_1$ can view herself as occupying vertex $(-(i-cn_1), -(j-dn_2))$, i.e. vertex $(-a,-a)$ with $a = i-cn_1$.  Similarly, there exist positive integers $c'$ and $d'$ such that $i'+c'n_1 = j'+d'n_2$, so $C_2$ can view herself as occupying vertex $(b,b)$ with $b = i'+_c'n_1$.  Thus $C_1$ and $C_2$ in fact already occupy vertices of the desired forms, and hence they can capture the robber.
\end{proof}

Lemma \ref{lem:rel_prime} allows us to determine the speed-$(s,s)$ cop number of the product of cycles for certain cycle lengths.
\begin{theorem}\label{thm:two_cycles}
Let $G = C_{n_1} \cart C_{n_2} \cart \dots \cart C_{n_{2k}}$.  If $n_{2i-1}$ and $n_{2i}$ are relatively prime for all $i \in \{1, \dots, k\}$, then $\speeds(G) = 2k$.  
\end{theorem}
\begin{proof}
The comments prior to Lemma \ref{lem:rel_prime} establish the lower bound.  For the upper bound, we use induction on $k$.  When $k=1$, the upper bound follows directly from Lemma \ref{lem:rel_prime}.  For $k > 1$, let $G' = C_{n_1} \cart C_{n_2} \cart \dots C_{n_{2k-2}}$; the induction hypothesis shows that $\speeds(G') = 2k-2$; now Lemma \ref{lem:active_and_change} shows that
\[\speeds(G) \le \speeds(G') + c^-(C_{2k-1} \cart C_{2k}) = 2k-2 + 2 = 2k,\]
where the fact that $c^-(C_{2k-1} \cart C_{2k}) = 2$ follows from Lemma \ref{lem:rel_prime}.
\end{proof}

In light of Theorem \ref{thm:two_cycles}, one might expect that always $\speeds(C_{n_1} \cart \dots \cart C_{n_k}) = k$ (provided that each $n_i$ is ``large'').  Perhaps unexpectedly, computational evidence shows that this is not the case; in particular, in the speed-$(2,2)$ game, the graphs $P_7 \cart P_7$, $P_8 \cart P_8$, and $P_9 \cart P_9$ all require three cops.  We suspect that perhaps $c_{2,2}(P_m \cart P_n) = 3$ whenever $m$ and $n$ have a sufficiently large common factor.
\end{section}

\begin{section}{Capture Time}\label{sec:capture_time}

In this section, instead of asking \textit{how many} cops are needed to capture a robber, we instead ask a different question: \textit{how quickly} can the cops capture the robber?   In particular, we are interested in graphs that are ``as bad as possible'' for the cops -- i.e. graphs that maximize the number of rounds needed for the cops to win.

We will focus on \textit{cop-win} graphs, that is, graphs for which one cop is sufficient to catch the robber.  Given a graph $G$ with $\speed{s}(G) = 1$, the \textit{speed-$(s,s)$ capture time} of $G$, denoted $\capt{s}(G)$, is the minimum number of rounds needed for the cop to guarantee capture of the robber in the speed-$(s,s)$ game on $G$.  (The initial placement of the cop does not count toward the capture time; for example, on $K_n$ the cop will always begin adjacent to the robber and capture him on the first cop turn, so $\capt{s}(K_n) = 1$.)  We are especially interested in the question of how large the speed-$(s,s)$ capture time can be among $n$-vertex cop-win graphs; we denote this quantity by $\capt{s}^*(n)$.

The notion of capture time (in the traditional model of Cops and Robbers) was first introduced by Bonato, Golovach, Hahn, and Kratochv\'{i}l \cite{BGHK09}, who proved that when $n \ge 5$ we have $\capt{1}^*(n)\in\{n-3,n-4\}$; Gaven\v{c}iak \cite{Gav10} later showed that in fact $\capt{1}^*(n)=n-4$.
%The Bonato et al. upper bound on $\capt{1}^*(n)$ relied on the construction of a family of graphs in which the robber can evade capture 
%with unique corners and showed that the robber could avoid being cornered for many turns (\textcolor{blue}{Is this okay to say? Didn't define ``being cornered", but I think people will get the gist.}).\comment{This is vague enough that if you didn't already know what was meant, I don't think it would make sense.}  
Moving from the original game to the speed-$(s,s)$ game for $s>1$, one might expect that $\capt{s}^*(n)$ is substantially less than $\capt{1}^*(n)$, as the cops are allowed to move faster and can  therefore capture the robber more swiftly.  In what follows, we investigate the case $s=2$ and show that, in fact, $\capt{2}^*(n)$ can only differ from $\capt{1}^*(n)$ by a small additive constant.  Specifically, we construct a family of graphs showing that for all $n \ge 9$, we have $\capt{2}^*(n)\ge n-7$.  

Before explaining the construction of these graphs, we introduce some terminology and properties of cop-win graphs that will be helpful moving forward. It is a simple consequence of the proof of Theorem~\ref{thm:graphpower} that for any speed $s$ and graph $G$ with $\speed{s}(G)=1$, we have $\capt{s}(G)=\capt{1}(G^s)$.  Similarly, it follows from the proof of Theorem~\ref{thm:speeds_factors_monotone} that for any speed $s$, any positive integer $k$, and any graph $G$ with $\speed{s}(G)=1$, we have $\capt{s}(G)\ge\capt{ks}(G)$; in particular, if $G$ is cop-win in the speed-$(1,1)$ game, then $\capt{1}(G) \ge \capt{2}(G)$. 
%\textcolor{blue}{Are these two assertions obvious? I.e. Can we omit the proofs?}\comment{The former assertion that $\capt{s}(G) \ge \capt{t}(G)$ is not obvious and I'm not sure it's generally true -- we don't even necessarily know that $G$ being cop-win in the speed-$s$ game implies that it's cop-win in the speed-$t$ game.  The new assertion is correct, but I'm not sure we'll need it.}  
%These facts will be useful in proving the main result of this section.  In order to capture the robber on a vertex $v$, the cop can occupy a vertex $u$ such that $N[u]\supseteq N[v]$.  We call such a vertex $v$ a \textit{corner} of the graph and say either that $u$ \textit{corners} $v$ or that $v$ \textit{is cornered by} $u$.

Since $\capt{2}(G) = \capt{1}(G^2)$, we can view the problem of determining $\capt{2}^*(n)$ as that of determining the maximum value of $\capt{1}(G^2)$ over all $n$-vertex graphs $G$ such that $G^2$ is cop-win in the original game.  Consequently, let us recall the characterization of cop-win graphs due to to Quillot \cite{Qui78} and to Nowakowski and Winkler \cite{NW83}:

\begin{definition}
   A vertex $v$ in a graph $G$ is called a \textit{corner} if, for some other vertex $u$, we have $N[u]\supseteq N[v]$.  In this case, we say that $u$ \textit{corners} $v$ or that $v$ \textit{is cornered by} $u$.
\end{definition}

\begin{theorem}[Quillot\cite{Qui78}; Nowakowski and Winkler \cite{NW83}]\label{thm:copwin}
    A graph $G$ on $n$ vertices is cop-win if and only if there is an ordering $\{v_1,\dots,v_n\}$ of the vertices of $G$ such that for all $i \in \{1, \dots, n-1\}$, vertex $v_i$ is a corner in the subgraph induced by $\{v_i,\dots,v_n\}$.  (When it exists, such an ordering is known as a \textit{cop-win ordering}.)
\end{theorem}

%\comment{I don't know that the following proposition is needed.}
%\begin{proposition}[\cite{***}]
%    Let $x$ be a corner in a graph $G$.  A graph $G$ is copwin if and only if $G-x$ is copwin.
%\end{proposition}

%\comment{I don't know that this is needed either.}As a consequence of the previous proposition, adding one corner to a copwin graph will result in another copwin graph.

In practice, cop-win orderings can be found by iteratively deleting corners from a graph until no vertices remain.  
%This process is called \textit{dismantling} the graph. \comment{ we don't use this term.} 
Clarke, Finbow, and MacGillivray \cite{CFM14} showed how to use this process to determine the capture time of the graph.
%, provided that, at each step of the process we identify vertices with have identical closed neighborhoods.  To formalize this, 
Their results involve a special equivalence relation on the vertices of the graph.  Given a graph $G$, we define the equivalence relation $\Theta_G$ on $V(G)$ by $(u,v)\in \Theta_G$ if and only if $N[u]=N[v]$.  Furthermore, we define $G / \Theta_G$  be the graph whose vertices are equivalence classes $\{[u]:u\in V(G)\}$ of $\Theta_G$ and $[u][v]\in E(G / \Theta_G)$ if and only if $uv\in E(G)$ and $v\not\in[u]$.  Intuitively, the equivalence classes of $\Theta_G$ are sets of ``twin'' vertices, i.e. vertices with equal neighborhoods; we want to be sure to treat all twin vertices the same, which is why we operate on equivalence classes of $\Theta_G$ rather than on individual vertices of $G$.

\begin{definition}[Clarke et al. \cite{CFM14}]
    Let $G$ be a cop-win graph that is not complete.  We can define an ordered partition $X_1,X_2,\dots,X_k$ of $V(G)$ by setting $G_i=G-\bigcup_{j=1}^{i-1}X_j$ (so that $G_i=G$), were $X_i$ denotes the set of corners of $G_i / \Theta_{G_i}$. We refer to $X_1,X_2,\dots,X_k$ as the \textit{cop-win partition} of $G$.
\end{definition}

We note that $X_{i+1},\dots,X_k$ is a cop-win partition of $G_i$.

\begin{proposition}[Clarke et al. \cite{CFM14}]\label{prop:copwin_copwin_partition}
    A graph is cop-win if and only if it has a cop-win partition.
\end{proposition}

%Clarke et al. \cite{CFM14} shows how to calculate the capture time in the following theorem.\comment{Not sure we need this line.}

\begin{theorem}[Clarke et al. \cite{CFM14}]\label{thm:capture_time_partition}
Let $G$ have a cop-win partition $X_1,X_2,\dots,X_k$. Then $\mathrm{capt}_{1,1}(G)= k-1$ if every vertex of $X_k$ is adjacent to every vertex of $X_{k-1}$, or if $G$ has only one vertex.  Otherwise, $\mathrm{capt}_{1,1}(G) = k$.
\end{theorem}

%Each of the results in (CITE BONATO) and (CITE GAVENCIAK) made use of a result like Lemma \ref{lem:retraction_capture_time} by constructing a gadget graph which had exactly one corner and appending the next vertex of a path to the unique corner of the current graph. (\textcolor{blue}{Need to say this previous sentence better}).

Armed with these results, we will build a family of graphs $\{G_n\}_{n=9}^\infty$ with large speed-$(2,2)$ capture time.  Each graph $G_n$ will contain $n$ vertices.  Among the $n$ vertices in $G_n$ will be two ``hub'' vertices, $h_1$ and $h_2$, along with special vertices $v_8, v_9, \dots, v_n$.  The graph $G_9$ is shown in the top left of Figure \ref{fig:capture_time_graphs}.  For $n\ge 10$, we construct $G_n$ from $G_{n-1}$ by adding a vertex $v_{n}$ adjacent only to $v_{n-1}$ and one of the hub vertices: if $n$ is odd then $v_n$ is adjacent to $h_1$, and otherwise it is adjacent to $h_2$.  See Figure \ref{fig:capture_time_graphs} for an illustration of the construction of $G_{10}, G_{11}$, and $G_{12}$.  

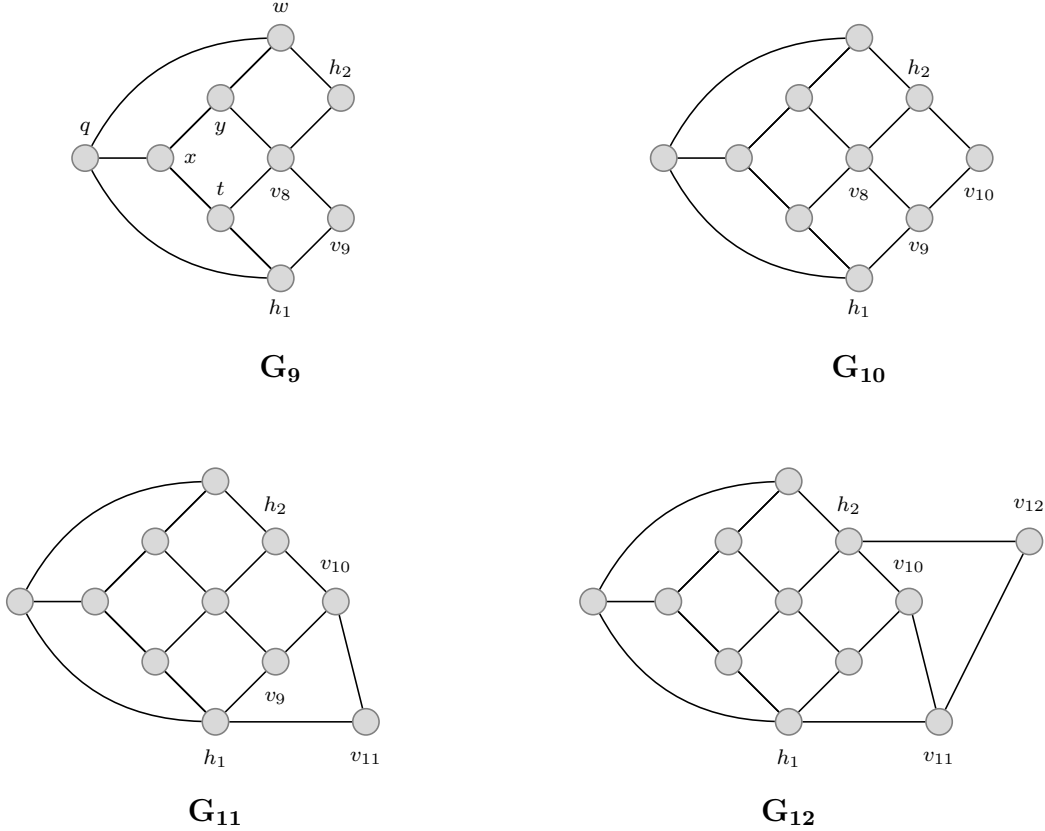
\begin{figure}[ht]
\hfill\,
\begin{tikzpicture}
[inner sep=0mm, semithick,
 vertex/.style={draw=gray, fill=gray!30, circle, minimum size=0.35cm},
 vertexlabel/.style={draw=none, fill=none, shape=rectangle, inner sep=2pt, font=\small},
 textnode/.style={draw=none, fill=none, shape=rectangle, inner sep=2pt},
 xscale=0.8,yscale=0.8,rotate=90]

\node (text) at (-0.5,-3) [textnode] {$\mathbf{G_{9}}$};
\node (v00) at (3,0.25) [vertex] {};
\node (v13) at (1,-3) [vertex] {};
\node (v22) at (2,-2) [vertex] {};
\node (v24) at (2,-4) [vertex] {};
\node (v31) at (3,-1) [vertex] {};
\node (v33) at (3,-3) [vertex] {};
\node (v42) at (4,-2) [vertex] {};
\node (v44) at (4,-4) [vertex] {};
\node (v53) at (5,-3) [vertex] {};

\draw (v13) -- (v24);
\draw (v44) -- (v53);
\draw (v24) -- (v33) -- (v42);
\draw (v22) -- (v33) -- (v44);
\draw (v13) -- (v22) -- (v31) -- (v42) -- (v53);
\draw (v13) -- (v22) -- (v31) -- (v42) -- (v53);
\draw (v00) to [bend right=30] (v13);
\draw (v00) -- (v31);
\draw (v00) to [bend left=30] (v53);

\node ($h_1$) at ($(v13)+(-0.5,0)$) [textnode] {\tiny \bfseries $h_1$};
\node ($h_2$) at ($(v44)+(0.5,0)$) [textnode] {\tiny \bfseries $h_2$};
\node ($v_{8}$) at ($(v33)+(-0.6,0)$) [textnode] {\tiny \bfseries $v_{8}$};
\node ($q$) at ($(v00)+(0.5,0)$) [textnode] {\tiny \bfseries $q$};
\node ($t$) at ($(v22)+(0.5,0)$) [textnode] {\tiny \bfseries $t$};
\node ($w$) at ($(v53)+(0.5,0)$) [textnode] {\tiny \bfseries $w$};
\node ($x$) at ($(v31)+(0,-0.5)$) [textnode] {\tiny \bfseries $x$};
\node ($y$) at ($(v42)+(-0.5,0)$) [textnode] {\tiny \bfseries $y$};
\node ($v_9$) at ($(v24)+(-0.5,0)$) [textnode] {\tiny \bfseries $v_9$};
\end{tikzpicture}
\hfill\hfill
\begin{tikzpicture}
[inner sep=0mm, semithick,
 vertex/.style={draw=gray, fill=gray!30, circle, minimum size=0.35cm},
 vertexlabel/.style={draw=none, fill=none, shape=rectangle, inner sep=2pt, font=\small},
 textnode/.style={draw=none, fill=none, shape=rectangle, inner sep=2pt},
 xscale=0.8,yscale=0.8,rotate=90]

\node (text) at (-0.5,-3) [textnode] {$\mathbf{G_{10}}$};
\node (v00) at (3,0.25) [vertex] {};
\node (v13) at (1,-3) [vertex] {};
\node (v22) at (2,-2) [vertex] {};
\node (v24) at (2,-4) [vertex] {};
\node (v31) at (3,-1) [vertex] {};
\node (v33) at (3,-3) [vertex] {};
\node (v35) at (3,-5) [vertex] {};
\node (v42) at (4,-2) [vertex] {};
\node (v44) at (4,-4) [vertex] {};
\node (v53) at (5,-3) [vertex] {};

\draw (v13) -- (v24) -- (v35) -- (v44) -- (v53);
\draw (v24) -- (v33) -- (v42);
\draw (v22) -- (v33) -- (v44);
\draw (v13) -- (v22) -- (v31) -- (v42) -- (v53);
\draw (v13) -- (v22) -- (v31) -- (v42) -- (v53);
\draw (v00) to [bend right=30] (v13);
\draw (v00) -- (v31);
\draw (v00) to [bend left=30] (v53);

\node ($h_1$) at ($(v13)+(-0.5,0)$) [textnode] {\tiny \bfseries $h_1$};
\node ($h_2$) at ($(v44)+(0.5,0)$) [textnode] {\tiny \bfseries $h_2$};
\node ($v_{10}$) at ($(v35)+(-0.6,0)$) [textnode] {\tiny \bfseries $v_{10}$};
\node ($v_{8}$) at ($(v33)+(-0.6,0)$) [textnode] {\tiny \bfseries $v_{8}$};
\node ($v_9$) at ($(v24)+(-0.5,0)$) [textnode] {\tiny \bfseries $v_9$};
\end{tikzpicture}
\hfill\,\\\bigskip\bigskip\bigskip
\hfill
\begin{tikzpicture}
[inner sep=0mm, semithick,
 vertex/.style={draw=gray, fill=gray!30, circle, minimum size=0.35cm},
 vertexlabel/.style={draw=none, fill=none, shape=rectangle, inner sep=2pt, font=\small},
 textnode/.style={draw=none, fill=none, shape=rectangle, inner sep=2pt},
 xscale=0.8,yscale=0.8,rotate=90]

\node (text) at (-0.5,-3) [textnode] {$\mathbf{G_{11}}$};
\node (v00) at (3,0.25) [vertex] {};
\node (v13) at (1,-3) [vertex] {};
\node (v22) at (2,-2) [vertex] {};
\node (v24) at (2,-4) [vertex] {};
\node (v31) at (3,-1) [vertex] {};
\node (v33) at (3,-3) [vertex] {};
\node (v35) at (3,-5) [vertex] {};
\node (v42) at (4,-2) [vertex] {};
\node (v44) at (4,-4) [vertex] {};
\node (v53) at (5,-3) [vertex] {};
\node (v15) at (1,-5.5) [vertex] {};

\draw (v13) -- (v24) -- (v35) -- (v44) -- (v53);
\draw (v24) -- (v33) -- (v42);
\draw (v22) -- (v33) -- (v44);
\draw (v13) -- (v22) -- (v31) -- (v42) -- (v53);
\draw (v13) -- (v22) -- (v31) -- (v42) -- (v53);
\draw (v00) to [bend right=30] (v13);
\draw (v00) -- (v31);
\draw (v00) to [bend left=30] (v53);
\draw (v13) -- (v15);
\draw (v35) -- (v15);

\node ($h_1$) at ($(v13)+(-0.6,0)$) [textnode] {\tiny \bfseries $h_1$};
\node ($h_2$) at ($(v44)+(0.6,0)$) [textnode] {\tiny \bfseries $h_2$};
\node ($v_{11}$) at ($(v15)+(-0.6,0)$) [textnode] {\tiny \bfseries $v_{11}$};
\node ($v_{10}$) at ($(v35)+(0.6,0)$) [textnode] {\tiny \bfseries $v_{10}$};
\node ($v_{9}$) at ($(v24)+(-0.6,0)$) [textnode] {\tiny \bfseries $v_{9}$};
\end{tikzpicture}
\hfill\hfill
\begin{tikzpicture}
[inner sep=0mm, semithick,
 vertex/.style={draw=gray, fill=gray!30, circle, minimum size=0.35cm},
 vertexlabel/.style={draw=none, fill=none, shape=rectangle, inner sep=2pt, font=\small},
 textnode/.style={draw=none, fill=none, shape=rectangle, inner sep=2pt},
 xscale=0.8,yscale=0.8,rotate=90]

\node (text) at (-0.5,-3) [textnode] {$\mathbf{G_{12}}$};
\node (v00) at (3,0.25) [vertex] {};
\node (v13) at (1,-3) [vertex] {};
\node (v22) at (2,-2) [vertex] {};
\node (v24) at (2,-4) [vertex] {};
\node (v31) at (3,-1) [vertex] {};
\node (v33) at (3,-3) [vertex] {};
\node (v35) at (3,-5) [vertex] {};
\node (v42) at (4,-2) [vertex] {};
\node (v44) at (4,-4) [vertex] {};
\node (v53) at (5,-3) [vertex] {};
\node (v15) at (1,-5.5) [vertex] {};
\node (v46) at (4,-7) [vertex] {};

\draw (v13) -- (v24) -- (v35) -- (v44) -- (v53);
\draw (v24) -- (v33) -- (v42);
\draw (v22) -- (v33) -- (v44);
\draw (v13) -- (v22) -- (v31) -- (v42) -- (v53);
\draw (v13) -- (v22) -- (v31) -- (v42) -- (v53);
\draw (v00) to [bend right=30] (v13);
\draw (v00) -- (v31);
\draw (v00) to [bend left=30] (v53);
\draw (v13) -- (v15);
\draw (v35) -- (v15);
\draw (v15) -- (v46);
\draw (v46) -- (v44);

\node ($h_1$) at ($(v13)+(-0.6,0)$) [textnode] {\tiny \bfseries $h_1$};
\node ($h_2$) at ($(v44)+(0.6,0)$) [textnode] {\tiny \bfseries $h_2$};
\node ($v_{12}$) at ($(v46)+(0.6,0)$) [textnode] {\tiny \bfseries $v_{12}$};
\node ($v_{11}$) at ($(v15)+(-0.6,0)$) [textnode] {\tiny \bfseries $v_{11}$};
\node ($v_{10}$) at ($(v35)+(0.6,0)$) [textnode] {\tiny \bfseries $v_{10}$};

\end{tikzpicture}
\hfill\,\\\bigskip
\caption{The graphs $G_9$, $G_{10}$, $G_{11}$, and $G_{12}$.}
\label{fig:capture_time_graphs}
\end{figure}

Intuitively, the key feature of these graphs is that each $G_n^2$ has exactly one corner, namely $v_n$.  Thus, to construct each $G_n$ from $G_{n-1}$, we add a unique corner.  Adding a unique corner increases a graph's capture time by 1 -- as we prove next, in Lemma \ref{lem:capture_time_unique_corner} -- so with each additional vertex we add, we increase the capture time. 
%A key feature of these graphs -- which we prove later, in Theorem \ref{thm:capture_time_speed2} -- is that each $G_n^2$ is cop-win and, moreover, has exactly one corner (namely $v_n$).  
%Our next lemma shows that removing a unique corner from a graph will reduce the capture time by 1; thus, each time we \textit{add} a unique corner 
%Lemma \ref{lem:capture_time_unique_corner} will help us to find the capture time of $G_n^2$ for all $n\ge 10$.

%\textcolor{blue}{Move four properties to Theorem \ref{thm:capture_time_speed2} and mention before Lemma \ref{lem:capture_time_unique_corner} that $G_n$ has a unique corner for all $n\ge 10$.}

%\comment{Maybe all of the below should simply be part of the proof that $\capt{2}^*(n) \ge n-7$.}

%Before proving that the $G_n$ satisfy properties (1)-(4), we introduce a helpful lemma.

\begin{lemma}\label{lem:capture_time_unique_corner}
    If $G$ is a nontrivial cop-win graph that contains a unique corner $v$, then $\mathrm{capt}_{1,1}(G)=\mathrm{capt}_{1,1}(G-v)+1$.
\end{lemma}
\begin{proof}
    Since $G$ is cop-win, by Proposition \ref{prop:copwin_copwin_partition}, $G$ has a cop-win partition $X_1,\dots,X_k$.  We note that there are no pairs of vertices $x,y\in V(G)$ such that $N[x]=N[y]$, as this would imply that both $x$ and $y$ are corners in $G$, contradicting the assumption that $v$ is the unique corner in $G$.  In particular, we have $G / \Theta_{G}\cong G$.  Furthermore, it follows that $X_1=\{v\}$.  Note that $G-v$ must have more than one vertex, since otherwise we would have $G=K_2$ and hence $G$ would have multiple corners.  Consequently, since $X_2,\dots,X_k$ is a cop-win partition of $G_2=G-X_1=G-v$, Theorem \ref{thm:capture_time_partition} shows that $\mathrm{capt}_{1,1}(G)=\mathrm{capt}_{1,1}(G-v)+1$, as desired.
\end{proof}

We are now ready to give a lower bound on $\capt{2}^*(n) \ge n-7$.

\begin{theorem}\label{thm:capture_time_speed2}
    For $n \ge 9$, we have $\capt{2}^*(n) \ge n-7$.
\end{theorem}
\begin{proof}
    To establish the theorem, it suffices to show that each $G_n$ is cop-win in the speed-$(2,2)$ game and that $\capt{2}(G_n) = n-7$.  To facilitate this, we will show that each $G_n$ (for $n \ge 10$) has the following properties:
\begin{enumerate}
    \item
        $v_{n-2}$ is the only vertex that corners $v_n$ in $G_n^2$;
    %\item
    %    $G_i$ is obtained from $G_{i-1}$ by attaching a vertex $v_i$ to $G_{i-1}$; \textcolor{blue}{Change this to definition of what $G_i$ is.}
    \item
        $v_n$ is the unique corner of $G_n^2$;
    \item
        $\speed{2}(G_n)=1$; and
    \item
        $\capt{2}(G_n)=n-7$.
\end{enumerate}

It is straightforward but tedious to verify that $G_9^2$ is cop-win and that $\mathrm{capt}_{1,1}(G_9^2)=2$.  (To aid the reader in this endeavor, we give the cop-win partition for $G_9$: $X_1=\{v_9,h_2\}$, $X_2=\{q,t,y,h_1,v_8,w\}$, $X_3=\{x\}$.  Note that all vertices in $X_2$ are adjacent to $x$; consequently, by Theorem \ref{thm:capture_time_partition}, we have $\mathrm{capt}_{1,1}(G_9^2)=3-1=2$.)

%    \comment{I see now what you meant about the properties being out of order.  (3) should really be (1), (2) should still be (2), (1) should be (3), and (4) should still be (4).  Maybe we should actually add a fifth property: $G_n^2-v_n = G_{n-1}^2$.  This doesn't really need induction, though...}
    For $n \ge 10$, we will show by induction that properties (1)-(4) hold (and in particular that $\capt{2}(G_n) = n-7$, as claimed).  It is clear by inspection and from the fact that $G_{10}^2-v_{10}=G_{9}^2$, as well as the result from Lemma \ref{lem:capture_time_unique_corner}, that all four properties hold when $n=10$.  Fix $k \ge 10$, and assume that properties (1)-(4) hold when $n \le k$; we will show that they also hold when $n=k+1$.  Let $n=k+1$ and suppose that $n$ is odd; the case where $n$ is even is similar.   We begin with properties (1) and (2).  Note that 
    \begin{align*}
    N_{G_n^2}[v_n] &= N_{G_n}[h_1] \cup N_{G_n}[v_{n-1}]\\
                   &= \{q,t,h_1,v_9,v_{11},\dots,v_{n-2},v_n\} \cup \{h_2,v_{n-2},v_{n-1},v_n\}\\
                   &= \{q,t,h_1,h_2,v_9,v_{11},\dots,v_{n-2},v_{n-1},v_n\};
    \end{align*}
    similarly, $N_{G_n^2}(v_{n-2})=\{q,t,h_1,h_2,v_9,v_{11},\dots,v_{n-4},v_{n-3},v_{n-1},v_n\}$.  Thus, $v_{n-2}$ corners $v_n$ in $G_n^2$.  Furthermore, $v_{n-2}$ is the only vertex in $G_n^2$ (aside from $v_n$) that is adjacent (in $G_n^2$) to $v_n, v_{n-1}$, $h_1$, $h_2$, and $q$.  
    %\comment{This may need justification.}  
    Therefore, $v_{n-2}$ is the only vertex that corners $v_n$ in $G_n^2$, which establishes property (1).  
    %Since (1) and (2) hold for $i-1$ \comment{This is not precise -- you mean ``(1) and (2) hold when $n=i-1$'', but really you should probably state the actual properties, e.g. ``by the induction hypothesis, $v_{n-1}$ is the unique corner in $G_{n-1}^2$ and is covered only by $v_{n-3}$.}, $v_{i-1}$ is the unique corner in $G_{i-1}^2$ and is only covered by $v_{i-3}$.  
    
    For (2), first note that $G_n^2-v_n = G_{n-1}^2$: all edges of $G_n^2-v_n$ are clearly present in $G_{n-1}^2$, with the possible exception of $h_1 v_{n-1}$ (since $h_1$ and $v_{n-1}$ are at distance 2 in $G_n$ by virtue of having $v_n$ as a common neighbor); however, $h_1$ and $v_{n-1}$ also have the common neighbor $v_{n-2}$ in $G_{n-1}$ and thus are in fact adjacent in $G_{n-1}^2$.  Consequently, all corners in $G_n^2$ (other than $v_n$) must necessarily be corners in $G_n^2-v_n$ and hence in $G_{n-1}^2$.  By the induction hypothesis, $v_{n-1}$ is the only corner in $G_{n-1}^2$ and it is only cornered by $v_{n-3}$; in $G_n^2$, however, the vertex $v_{n-1}$ is adjacent to $v_n$ but $v_{n-3}$ is not.  Thus $v_{n-3}$ does not corner $v_{n-1}$ in $G_n^2$, and so $v_{n-1}$ is not a corner in $G_n^2$.  It follows that $v_n$ is the unique corner in $G_n^2$, which establishes property (2).   
    
    It still remains to establish properties (3) and (4).  As argued above, $v_n$ is a corner in $G_n^2$, and $G_n^2-v_n = G_{n-1}^2$.  By the induction hypothesis, $G_{n-1}^2$ is cop-win; since removing a corner from $G_n^2$ yields a cop-win graph, $G_n^2$ itself must be cop-win, which establishes (3).  Finally, by Theorem \ref{thm:capture_time_partition}, we have $\capt{2}(G_n)=\mathrm{capt}(G_n^2)=\mathrm{capt}(G_n^2-v_n)+1=\mathrm{capt}(G_{n-1}^2)+1=\left((n-1)-7\right)+1=n-7$, which establishes (4).
\end{proof}

Theorem \ref{thm:capture_time_speed2} shows that when $n \ge 9$, we have $\capt{2}^*(n) \ge n-7$.  While we make no claim that in fact $\capt{2}^*(n) = n-7$, we do remark that based on computer search, it seems \textit{likely} that this is the case.  Moreover, if $G$ is cop-win in the speed-$(2,2)$ game, then $G^2$ is cop-win in the speed-$(1,1)$ game, and hence $\capt{2}(G) = \capt{1}(G^2) \le n-4$ by the result of Gaven\v{c}iak.  Consequently, for $n \ge 9$, we have $\capt{2}^*(n) \in \{n-7, n-6, n-5, n-4\}$.  Since every cop-win graph contains a corner, and since removing a corner decreases the capture time by at most 1, if one could show (perhaps by computer search) that $\capt{2}^*(9) = 2$ or that $\capt{2}^*(10) = 3$, then it would follow that $\capt{2}^*(n) \le n-7$.

\end{section}

\section{Future Work}\label{sec:open_questions}

Our work suggests several open problems and directions for future research; we list a few of the most interesting ones here.\\

\begin{itemize}
\item Theorem \ref{thm:subdivision} states that for every graph $G$, we have $c(G) \le \speeds(G^{(s)}) \le c(G)+1$.  However, although we proved that the upper bound holds with equality when $G=K_n$, we have been unable to find any examples of graphs $G$ such that $c_{s,s}(G^{(s)})=c(G)+1$ and $c(G)\ge 2$.  Are there graphs with $\speeds(G^{(s)}) = c(G)+1$ and $c(G) = k$ for all $s \ge 2$ and $k \ge 2$?\\  %\comment{The following is OK for the thesis, since it explains how you've tried approaching the problem, but I intend to remove it before submitting the paper for publication.} It is clear from following the proof of Theorem \ref{thm:subdivision} that $c(G)$ cops can capture a branch vertex in $G^{(s)}$ that is within distance $s$ of the robber.  From that point onward, the robber can be relegated to only occupying subdivision vertices for the remainder of the game by having one cop follow each of the robber's moves.  However, it is unclear how the remaining $c(G)-1$ cops should proceed in order to corner the robber.\\

\item Theorems \ref{thm:speeds_factors_monotone} and \ref{thm:nonincreasing_sequence} give some evidence that $c_{s,s}$ is monotonically decreasing in $s$ for any graph $G$ (see Conjecture \ref{conj:monotone}).  We suspect that this is indeed the case, but a proof eludes us.\\

\item Theorem \ref{thm:projective_plane_product} shows that in general $\speeds(G\cart H)$ cannot be bounded above by any function of $\speeds(G)$ and $\speeds(H)$.  However, is it possible to establish an upper bound on $\speed{s}(G \cart H)$ in terms of $\speed{t}(G)$ and $\speed{s-t}(H)$ for some $t < s$?  In particular, is it the case that
$$\speed{s}(G \cart H) \le \speed{t}(G) + \speed{s-t}(H)$$
for all $t \in \{1, \dots, s-1\}$?\\

\item We showed that the upper bound for $c_{s,s}(Q_d)$ given in Theorem \ref{thm:hypercube_upper} is tight when $s=2$.  In fact, Corollary \ref{cor:speed_2_hypercube} categorizes $c_{2,2}(Q_d)$ for all $d$.  For $s>2$, in order to match the upper bound, one would need to show that the robber can evade need $k$ cops on $Q_d$ when $d=2s+2k-2$.  We suspect this to be true, but proving it seems surprisingly difficult.  In particular, the technique used to prove Theorem \ref{thm:speed_2} does not seem to extend to large $s$, and the techniques behind Theorem \ref{thm:speed_s_hypercube_lower_bound} do not seem to be viable unless $s \gg k$.  Establishing a tight lower bound on $\speed{s}(Q_d)$ would seem to require a new approach.\\

\item For speed $s=2$, by Theorem \ref{thm:speed_s_large_grids} and the proof of Theorem \ref{thm:speed_2}, we have that if $G$ is the Cartesian product of $d$ trees, then $\left\lfloor d/2\right\rfloor\le c_{2,2}(G)\le\left\lceil d/2\right\rceil$.  It would be interesting to characterize the choices of $T_i$ that make $\speedtwo(G)$ equal to $\floor{d/2}$, and those that make it equal to $\ceil{d/2}$.\\

\item We suspect that $\capt{2}^*(n)=n-7$ for all $n\ge 9$.  By the result in Theorem \ref{thm:capture_time_speed2}, one only need show that $\capt{2}^*(9)=2$ or that $\capt{2}^*(10)=3$.  (Doing this by hand seems tedious, but computer search may be viable.)\\

\item From computer search of small graphs, we believe that $\capt{3}^*(n)=n-O(1)$; it would be interesting to verify (or refute) this.  More generally, what are the asymptotics of $\capt{s}^*(n)$ for larger $s$?  Is it perhaps the case that, in general, $\capt{s}^*(n) = n - f(s)$ for some function $f$?  
%We can construct graphs with relatively high capture time in the speed-$(3,3)$ game, but not quite reaching $O(n)$.  An interesting project would be to determine the asymptotics of $\capt{s}^*(n)$ for all $s\ge3$.  For which values of $s$ is $\capt{s}^*(n)=\Theta(n)$?
\end{itemize}
\end{section}

\end{document}